\documentclass[reqno,12pt]{amsart} 
\usepackage[a4paper,top=2.5cm,bottom=2.5cm,left=3cm,right=3cm]{geometry}
\usepackage{mathtools} 
\usepackage{amsfonts,amssymb,amsthm,stmaryrd}
\usepackage{mathrsfs} 
\usepackage{bbold} 
\usepackage{xfrac} 
\usepackage{stackengine} 
\usepackage{bm} 
\usepackage{enumitem}
\usepackage{graphicx}
\usepackage{float} 
\usepackage{tikz} 
\usepackage{array}
\usepackage{tabularray}
\usepackage{multirow} 
\usepackage{xcolor}
\usepackage{csquotes} 
\usepackage{hyperref}
\usepackage{cleveref} 
\usepackage{comment}
\usepackage[toc,page]{appendix}

\title{Two-point correlations of multiplicative functions with dense orbits}

\author[N. Tardy]{N\'eo Tardy}
\address{DER de math\'ematiques, ENS Paris-Saclay, 91190 Gif-sur-Yvette, France }
\email{neo.tardy@ens-paris-saclay.fr}

\date{}

\newtheorem{theoremalpha}{Theorem}

\newtheorem{theorem}{Theorem}[section]
\newtheorem{proposition}[theorem]{Proposition}
\newtheorem{lemma}[theorem]{Lemma}
\newtheorem{corollary}[theorem]{Corollary}
\newtheorem{conjecture}[theorem]{Conjecture}

\theoremstyle{definition}
\newtheorem{definition}[theorem]{Definition}

\theoremstyle{remark}
\newtheorem*{remark}{Remark}

\numberwithin{equation}{section}

\begin{document}

\maketitle

\begin{abstract}
    Let $f,g:\mathbb{N}\to\mathbb{T}$ be completely multiplicative functions with dense images in the complex unit circle $\mathbb{T}$. We prove that, for every non-empty open set $U \subset \mathbb{T}^2$, the set of integers $n$ such that $(f(n),g(n+1)) \in U$ has positive upper logarithmic density, unless the pair $(f,g)$ is of a special form. This result strengthens earlier theorems of Klurman and Mangerel, as well as of Charamaras, Mountakis and Tsinas, and yields simpler proofs.
\end{abstract}

\section{Introduction}

Denote by $\overline{\mathcal{M}}$ the set of completely multiplicative functions $f:\mathbb{N}\to\mathbb{T}$ such that the image $\{f(n)\}_n$ is dense in the unit circle $\mathbb{T}=\{z \in \mathbb{C} : |z|=1\}$. Let $f \in \overline{\mathcal{M}}$ and $n \ge 1$. Then, $f(n)$ is determined by its values at prime divisors of $n$ via the relation $f(n)=\prod_{p^\alpha||n} f(p)^\alpha$. Similarly, let $g \in \overline{\mathcal{M}}$, then $g(n+1)$ is determined by the values of $g(p)$ for all primes $p|n+1$. Since $n$ and $n+1$ are coprime, one might expect $f(n)$ and $g(n+1)$ to behave independently. In particular, since the sets $\{f(n)\}_n$ and $\{g(n+1)\}_n$ are dense in $\mathbb{T}$, one might expect the set $\{(f(n),g(n+1))\}_n$ to be dense in $\mathbb{T}^2$, which was the subject of a conjecture of Dar\'oczy and K\'atai \cite{daroczy1989characterization} in 1989. However, this heuristic fails in general, and one can construct explicit counterexamples.
\begin{itemize}
    \item Let $p$ be prime and $\alpha,\beta$ two irrational numbers. Define $f,g \in \overline{\mathcal{M}}$ by $f(p)=e(\alpha)$, $g(p)=e(\beta)$ and $f(q)=g(q)=1$ for all prime $q \neq p$. Then, for all integers $n \ge 1$, either $f(n)$ or $g(n+1)$ is equal to $1$. We denote the set of such pairs $(f,g)$ by $\mathcal{F}$.
    \item Let $t$ be a nonzero real number. Define $f,g \in \overline{\mathcal{M}}$ by $f(n)=g(n)=n^{it}$. Then, $g(n+1)=f(n)+o(1)$. We denote the set of such pairs $(f,g)$ by $\mathcal{G}$.
    \item Let $(f,g) \in \overline{\mathcal{M}}^2$ such that there exist  $K_1,K_2 \in \mathbb{N}$ such that $(f^{K_1},g^{K_2}) \in \mathcal{F} \cup \mathcal{G}$. By replacing $K_1$ and $K_2$ by $K_1 K_2$, we can assume that $K_1=K_2=K$. If $(f^K,g^K) \in \mathcal{F}$, then for all integers $n \ge 1$, either $f(n)$ or $g(n+1)$ is a $K$-th root of unity. If $(f^K,g^K) \in \mathcal{G}$, then for $n$ sufficiently large, the distance of $g(n+1)$ from the set $\{ f(n)e\left(\frac{m}{K}\right) : 0 \le m \le K-1 \}$ is at most $\frac{1}{3K}$. Consequently, in both cases, we have $\overline{\{(f(n),g(n+1))\}_n} \neq \mathbb{T}^2$.
\end{itemize}
The following theorem due to Klurman and Mangerel~\cite{klurman2018orbits} states that when the pair $(f,g) \in \overline{\mathcal{M}}^2$ is not one of the above examples, then $\overline{\{(f(n),g(n+1))\}_n} = \mathbb{T}^2$.

\begin{theoremalpha}\label{thm:A}
    Let $f,g \in \overline{\mathcal{M}}$ such that for all $K \ge 1$, $(f^K,g^K) \notin \mathcal{F}$. Assume that $\overline{\{(f(n),g(n+1))\}_n} \neq \mathbb{T}^2$, i.e. there exist $z,w\in\mathbb{T}$ and $\varepsilon>0$ such that
    \begin{equation}\label{thm:A-0}
        \forall x \ge 1, \mathcal{A}(x) = \emptyset,
    \end{equation}
    where
    $$ \mathcal{A}(x) = \mathcal{A}\left(z,w,\varepsilon;f,g;x\right) \vcentcolon= \left\{ n\le x : \left|f(n)-z\right|, \left|g(n+1)-w\right| \le \varepsilon \right\}. $$
    Then, there exist $K_1,K_2\in\mathbb{N}$ and a nonzero real number $T$ such that $f^{K_1}(n)=g^{K_2}(n)=n^{iT}$.
\end{theoremalpha}

Note that \eqref{thm:A-0} may be written as $d_{\mathrm{top}}(\mathcal{A})=0$, where
$$d_{\mathrm{top}}(\mathcal{A}) := \mathbb{1}_{\mathcal{A}(\infty)\neq\emptyset}.$$
We prove a quantitative refinement of Theorem~\ref{thm:A} by replacing the topological density $d_{\mathrm{top}}(\mathcal{A})$ of $\mathcal{A}$ by the \textit{logarithmic density}
$$d_{\mathrm{log}}(\mathcal{A}) := \lim\limits_{x\to+\infty}{\frac{1}{\log{x}}\sum_{n\in\mathcal{A}(x)} \frac{1}{n}},$$
if it exists, which is sensitive to the analytic structure of $\mathcal{A}$. This answers a conjecture of De Koninck, K\'atai and Phong in the affirmative.

\begin{theorem}[{\cite[Conjecture 1]{de2021variations}}]\label{thm:main}
    The statement of Theorem~\ref{thm:A} holds when one replaces \eqref{thm:A-0} by
    \begin{equation*}\label{thm:main-0}
        d_{\mathrm{log}}(\mathcal{A})=0.
    \end{equation*}
\end{theorem}

\begin{remark}
    Our proof of Theorem~\ref{thm:main} significantly simplifies the proof of Theorem~\ref{thm:A} in \cite{klurman2018orbits} by avoiding the use of Szemerédi's theorem. Instead, we rely on elementary techniques involving only standard arithmetic.
\end{remark}

\begin{remark}
    The statement of Theorem~\ref{thm:main} remains valid when one considers the pair $(f(an+b),g(cn+d))$, where $a,b \ge 1$ and $c,d \ge 0$ are integers satisfying $ad-bc \neq 0$. Indeed, the proof requires only a minor modification of the parameters, since the arithmetic structure of $a,b,c,d$ affects only finitely many prime powers. More precisely, our definition of $\mathcal{F}$ suggests that, compared with the setting of Theorem~\ref{thm:main}, the only additional sources of rotation on each coordinate of $\mathbb{T}^2$ arise from the divisors of $ad-bc$. From this observation, it is not difficult to see that Theorem~\ref{thm:main} implies the result of Charamaras, Mountakis and Tsinas \cite{charamaras2025multiplicative}, which characterizes the integers $a,b,c,d$ such that, for every completely multiplicative function $f:\mathbb{N}\to\mathbb{T}$,
    $$ \liminf_{n\to+\infty} |f(an+b)-f(cn+d)| = 0. $$
    This in turn generalizes earlier results of Klurman and Mangerel \cite{klurman2018rigidity}, corresponding to $a=c=d=1$ and $b=0$, and of Donoso, Le, Moreira and Sun \cite[Corollary 1.7]{donoso2023additive}, corresponding to $a=c$ and $b=0$. Using ergodic methods, Leung and T\'{a}fula \cite[Theorem 1.2]{leung2024multiplicative} recently extended these results to arbitrary finite families of completely multiplicative functions.
\end{remark}

It is easy to see that Theorem~\ref{thm:main} implies Theorem~\ref{thm:A}, since if $d_{\mathrm{top}}(\mathcal{A}) = 0$, then $d_{\mathrm{log}}(\mathcal{A})$ exists and is equal to $0$. Moreover, it is interesting to ask what happens when one replaces the logarithmic density in Theorem~\ref{thm:main} with other standard analytic densities, such as the \textit{natural density}
$$ d(\mathcal{A}) := \lim\limits_{x\to+\infty}{\frac{1}{x}\sum_{n\in\mathcal{A}(x)} 1}, $$
the \textit{lower natural density}
$$ \underline{d}(\mathcal{A}) := \liminf\limits_{x\to+\infty}{\frac{1}{x}\sum_{n\in\mathcal{A}(x)} 1}, $$
or the \textit{lower logarithmic density}
$$ \underline{d}_{\log}(\mathcal{A}) := \liminf\limits_{x\to+\infty}{\frac{1}{\log x}\sum_{n\in\mathcal{A}(x)} \frac{1}{n}}. $$
For the natural density, we have the following immediate corollary.

\begin{corollary}\label{cor:otherdensities}
    The statement of Theorem~\ref{thm:A} holds when one replaces \eqref{thm:A-0} by
    $$ d(\mathcal{A})=0 $$
\end{corollary}
\begin{proof}
    If $d(\mathcal{A})=0$, then it is a well-known fact that by partial summation, $d_{\mathrm{log}}(\mathcal{A})$ exists and equals $0$. Thus, we can apply Theorem~\ref{thm:main}.
\end{proof}

If one replaces \eqref{thm:A-0} by $\underline{d}(\mathcal{A})=0$ or $\underline{d}_{\log}(\mathcal{A})=0$, it appears that the conclusion of Theorem~\ref{thm:A} fails, as a consequence of the following proposition.

\begin{proposition}\label{prop:counterex}
    There exists $f \in \overline{\mathcal{M}}$ such that for all $K \in \mathbb{N}$, $(f^K,f^K) \notin \mathcal{F} \cup \mathcal{G}$, and
    \begin{equation}\label{prop:counterex-0-1}
        \liminf\limits_{x\to+\infty} \frac{1}{\log x} \sum_{\substack{n \le x \\ |f(n)\overline{f(n+1)}+1| \le \frac{1}{3}}} \frac{1}{n} = 0.
    \end{equation}
    In particular,
    \begin{equation}\label{prop:counterex-0-2}
        \liminf\limits_{x\to+\infty} \frac{1}{x} \sum_{\substack{n \le x \\ |f(n)\overline{f(n+1)}+1| \le \frac{1}{3}}} 1 = 0. 
    \end{equation}
\end{proposition}

Finally, we state two corollaries of our main theorem. The first corollary establishes a conjecture of De Koninck, K\'atai and Phong \cite[Conjecture 4]{de2019some} in the case of completely multiplicative functions, and states that there is no example of $f$ in Proposition~\ref{prop:counterex} for which the limit inferior is a limit.

\begin{corollary}\label{cor:DKP2}
    Let $f:\mathbb{N}\to\mathbb{T}$ be a completely multiplicative function such that there exist some $w \in \mathbb{T}$ and some $\varepsilon>0$ such that
    $$ \lim\limits_{x\to+\infty} \frac{1}{\log x} \sum_{\substack{n \le x \\ |f(n)\overline{f(n+1)}-w| \le \varepsilon}} \frac{1}{n} = 0. $$
    Then, there exist $K \ge 1$ and $t \in \mathbb{R}$ such that $f^K(n)=n^{it}$ for all $n \in \mathbb{N}$.
\end{corollary}

\begin{corollary}\label{cor:katai2}
    Let $f,g:\mathbb{N}\to\mathbb{T}$ be two completely multiplicative functions. Assume that
    $$ \lim\limits_{x\to+\infty} \frac{1}{\log x} \sum_{n \le x} \frac{|f(n)-g(n+1)|}{n} = 0. $$
    Then, $f(n)=g(n)=n^{it}$ for some $t \in \mathbb{R}$.
\end{corollary}

\subsection*{Structure of the paper}

In order to prove Theorem~\ref{thm:main}, we rely on the following guiding principle: a multiplicative function $f:\mathbb{N}\to\mathbb{T}$ is either close to some $n^{it}$, where $t \in \mathbb{R}$ (we say that $f$ is pretentious), in which case it behaves like the structured function $n^{it}$, or it is close to none of the $n^{it}$ (we say that $f$ is non-pretentious), in which case it behaves like a random function. In the pretentious case, the position of $f(n)$ on $\mathbb{T}$ can be controlled through the size of $n$. In the non-pretentious case, we exploit the arithmetic structure of the integers to rotate $f(n)$ towards a prescribed direction (this is called the \textit{rotation trick}, and was introduced by Klurman, Mangerel, Pohoata and Ter{\"a}v{\"a}inen in \cite{klurman2021multiplicative}). This philosophy is motivated by Hal\'asz's theorem, which asserts that the mean value $\frac{1}{x}\sum_{n \le x} f(n)$ of $f$ is close to $\frac{1}{x}\sum_{n \le x} n^{it}$ if $f$ is close to $n^{it}$, and  tends to $0$ otherwise. This theorem is one of the cornerstones of pretentious multiplicative number theory, developed by Granville and Soundararajan in \cite{granvillemultiplicative}. We make an extensive use of tools from pretentious theory, which we collect in Section~\ref{sec:pretentious}.

In Section~\ref{sec:lemmas}, we establish a few technical lemmas on the distribution of multiplicative functions on the unit circle.

In Sections \ref{sec:reduction} and \ref{sec:proof}, we prove Theorem~\ref{thm:main}. In Section~\ref{sec:reduction}, we first reduce to the case where $f^{K_1}$ and $g^{K_2}$ are both pretentious for some integers $K_1,K_2 \ge 1$. We then complete the proof in the pretentious setting in Section~\ref{sec:proof}. To this end, we assume, without loss of generality, that $f^{K_1}$ is close to $n^{it}$ but does not coincide with it. In particular, there exists a prime $p$ such that $\frac{f(p)}{p^{it}}$ is an irrational rotation of the circle. The powers of this ratio provide an independent source of density, ultimately leading to a contradiction.

In Section~\ref{sec:counterex}, we prove Proposition~\ref{prop:counterex}, establishing the limits of possible generalizations of Theorem~\ref{thm:A}.

Finally, in Section~\ref{sec:corollaries}, we prove Corollaries \ref{cor:DKP2} and \ref{cor:katai2}.

\subsection*{Notation}

We denote by $\mathbb{D}$ (resp. $\mathbb{T}$) the closed unit disk (resp. unit circle) of $\mathbb{C}$. We write the elements of $\mathbb{T}$ as $e(\theta) \vcentcolon= e^{2\pi i \theta}$, where $\theta \in \mathbb{R}$. Let $K \ge 1$. We denote by $\mu_K \vcentcolon= \{z \in \mathbb{C} : z^K=1\}$ the set of $K$-th roots of unity. Let $z\in\mathbb{T}$ and $\rho\in(0,\pi]$, we write $B_\rho(z) \subset \mathbb{T}$ for the open arc of length $4\pi\rho$ with center $z$.

Throughout this paper, $\mathcal{M}$ denotes the set of multiplicative functions $f:\mathbb{N}\to\mathbb{D}$. We denote by $\overline{\mathcal{M}}$ the subset of completely multiplicative functions $f:\mathbb{N}\to\mathbb{T}$ whose images are dense in $\mathbb{T}$.

If $\chi$ is a Dirichlet character mod $q$, we associate the $\mathbb{T}$-valued completely multiplicative function $\widetilde{\chi}$ by replacing all its zero values at primes with $1$:
$$
\forall p, \widetilde{\chi}(p) =
\begin{cases}
    1 & \ \mathrm{if} \ p|q, \\
    \chi(p) & \ \mathrm{if} \ p \nmid q.
\end{cases}
$$

We use the standard Vinogradov notation $\ll$, $\gg$ and $\asymp$.

Let $n \ge 1$, we denote by $P^-(n)$ (resp. $P^+(n)$) the minimal (resp. maximal) prime factor of $n$, with the convention that $P^-(1)=1$ (resp. $P^+(1)=0$).

\subsection*{Acknowledgments}

I am grateful to Oleksiy Klurman who introduced me to pretentious theory and guided me throughout this project. I would also like to thank Besfort Shala for our insightful conversations and Ioannis Kousek for pointing out an error in an earlier version of this paper. Finally, I thank the University of Bristol for providing excellent working conditions.

\section{Background in pretentious theory}\label{sec:pretentious}

In the proof of Theorem~\ref{thm:main}, we interpret the condition $f(n) \in B_\varepsilon(z)$ through the lens of pretentious theory. To be able to apply Hal\'asz-type theorems stated later in this section, we first need to approximate the indicator function of an arc of $\mathbb{T}$ by a Fourier series.

\begin{lemma}\label{lem:fourierapprox}
    Let $\varepsilon \in (0,\frac{1}{3})$ and $\delta\in\left(0,\frac{\varepsilon}{2}\right)$. Then for all $z, u\in\mathbb{T}$ such that $\frac{u}{z}\notin B_\delta\left(e\left(\pm\varepsilon\right)\right)$, and $X\geq1$, we have
    $$ \mathbb{1}_{B_\varepsilon(z)}(u)=\sum_{\left|n\right|\le X}{\frac{\sin{\left(2\pi n\varepsilon\right)}}{\pi n}\left(\frac{u}{z}\right)^n}+O\left(\frac{1}{\delta X}\right) $$
    with the convention that for $n=0$, $\frac{\sin{\left(2\pi n\varepsilon\right)}}{\pi n} = 2\varepsilon$.
\end{lemma}
\begin{proof}
    By replacing $(z,u)$ with $\left(1,\frac{u}{z}\right)$, we may assume that $z=1$. For $n \in \mathbb{Z}$, write $c_n\left(\mathbb{1}_{B_\varepsilon(1)}\right)$ the Fourier coefficients of $\mathbb{1}_{B_\varepsilon(1)}$. Then,
    $$ c_n\left(\mathbb{1}_{B_\varepsilon(1)}\right)=\int_{-\varepsilon}^{\varepsilon}{e\left(-ns\right)\mathrm{d}s}=\frac{\sin{\left(2\pi n\varepsilon\right)}}{\pi n}. $$
    It follows from the Dirichlet--Jordan test that
    $$ \mathbb{1}_{B_\varepsilon(1)}(u)=\sum_{n\in\mathbb{Z}}{c_n\left(\mathbb{1}_{B_\varepsilon(1)}\right)u^n}=\sum_{\left|n\right|\le X}{\frac{\sin{\left(2\pi n\varepsilon\right)}}{\pi n}u^n}+R, $$
    where
    $$ R \vcentcolon= \frac{2}{\pi}\sum_{n>X}\frac{\sin{\left(2\pi n\varepsilon\right)}\cos{\left(2\pi n\theta\right)}}{n}, $$
    with $u=e(\theta)$ ($\left|\theta\right|\le\frac{1}{2}$ and $\theta \notin [-\varepsilon-\delta,-\varepsilon+\delta]\cup[\varepsilon-\delta,\varepsilon+\delta]$). By summation by parts, we have
    \begin{equation}\label{lem:fourierapprox-1}
        R\ll\frac{1}{X}\sum_{1\le n\le X}{\sin{\left(2\pi n\varepsilon\right)}\cos{\left(2\pi n\theta\right)}}+\sum_{N>X}{\frac{1}{N^2}\sum_{1\le n\le N}{\sin{\left(2\pi n\varepsilon\right)}\cos{\left(2\pi n\theta\right)}}}.
    \end{equation}
    In addition, for all $x\geq1$, we have
    \begin{equation}\label{lem:fourierapprox-2}
        \sum_{1\le n\le x}{\sin{\left(2\pi n\varepsilon\right)}\cos{\left(2\pi n\theta\right)}}=\frac{1}{2}\sum_{1\le n\le x}\left(\sin{\left(2\pi n\left(\theta+\varepsilon\right)\right)}-\sin{\left(2\pi n\left(\theta-\varepsilon\right)\right)}\right).
    \end{equation}
    Note that for all $0<\left|\alpha\right|\le\frac{1}{6}$, we have
    $$ \sum_{1\le n\le x}\sin{\left(2\pi n\alpha\right)}=\frac{-1}{2\sin{\left(2\pi\alpha\right)}}\sum_{1\le n\le x}\left(\cos{\left(2\pi(n+1)\alpha\right)}-\cos{\left(2\pi(n-1)\alpha\right)}\right)\ll\frac{1}{\left|\alpha\right|}, $$
    and that for all $\frac{1}{6}<\left|\alpha\right|<\frac{5}{6}$, writing $\alpha=\frac{1}{2}+\eta$, we have
    \begin{equation*}
        \begin{aligned}
            \sum_{1\le n\le x}\sin{\left(2\pi n\alpha\right)}
            &= \sum_{1\le n\le x}{\left(-1\right)^n\sin{\left(2\pi n\eta\right)}} \\
            &= \frac{1}{2\sin{\left(2\pi\eta\right)}}\sum_{1\le n\le x}{\left(-1\right)^{n-1}\left(\cos{\left(2\pi(n+1)\eta\right)}-\cos{\left(2\pi(n-1)\eta\right)}\right)} \\
            &\ll \frac{1}{\left|\eta\right|}\left(\left|\cos{\left(2\pi\left(\lfloor x\rfloor+1\right)\eta\right)}-\cos{\left(2\pi\lfloor x\rfloor\eta\right)}\right|+\left|\cos{(2\pi\eta)}-1\right|\right) \\
            &\ll 1.
        \end{aligned}
    \end{equation*}
    We deduce that \eqref{lem:fourierapprox-2} is
    $$ \ll\frac{1}{\left|\theta+\varepsilon\right|}+\frac{1}{\left|\theta-\varepsilon\right|}\ll\frac{1}{\delta}, $$
    hence $R\ll\frac{1}{\delta X}$ by \eqref{lem:fourierapprox-1}.
\end{proof}

We now introduce some tools from pretentious multiplicative number theory. We recall the basic notions and refer to \cite{granvillemultiplicative} for further details. Let $x \in [1,\infty]$, we define the \textit{pretentious distance} up to $x$ on the set $\mathcal{M}$ of multiplicative functions taking values in the unit disk as follows:
$$ \mathbb{D}^2(f,g;x) := \sum_{p \le x} \frac{1-\mathfrak{R}(f(p)\overline{g(p)})}{p} \in [0,\infty]. $$
It satisfies the triangle inequalities
$$ \mathbb{D}(f,h;x) \le \mathbb{D}(f,g;x) + \mathbb{D}(g,h;x) $$
and
$$ \mathbb{D}(f_1 f_2,g_1 g_2;x) \le \mathbb{D}(f_1,g_1;x) + \mathbb{D}(f_2,g_2;x). $$
The latter inequality is sometimes called the \textit{second triangle inequality}. Another important property is that, if $F$ denotes the Dirichlet series associated with $f$, then for all $x \ge 2$ and $t \in \mathbb{R}$, we have
$$
\left| F\left( 1+\frac{1}{\log x}+it \right) \right|
\asymp (\log x) \left| \sum_{k=0}^{\infty} \frac{f(2^k)}{2^{\left(1+\frac{1}{\log x}+it\right)k}} \right| \exp \left( -\mathbb{D}^2(f,n^{it};x) \right).
$$
Let $1 \le y \le x$, we also define
$$ \mathbb{D}^2(f,g;y,x) := \sum_{y < p \le x} \frac{1-\mathfrak{R}(f(p)\overline{g(p)})}{p}. $$

\begin{definition}
    Let $f,g \in \mathcal{M}$. We say that $f$ \textit{pretends to} $g$ if $\mathbb{D}(f,g;\infty)<\infty$. We say that $f$ is \textit{pretentious}\footnote{In this paper, a pretentious function is a function that pretends to some $\chi n^{it}$. This notion may vary in other contexts.} if there exist a Dirichlet character $\chi$ and a real number $t$ such that $f$ pretends to $\chi n^{it}$. Otherwise, we say that $f$ is \textit{non-pretentious}. Finally, we say that $f$ is \textit{pseudo-pretentious} if there exists a positive integer $k$ such that $f^k$ is pretentious. Otherwise, we say that $f$ is \textit{non-pseudo-pretentious}.
\end{definition}

We will use the following version of Halász's theorem in arithmetic progressions.

\begin{lemma}\label{lem:loghalasz}
    Let $f:\mathbb{N} \to \mathbb{D}$ be completely multiplicative and non-pretentious, and $1 \le a \le Q$ such that $(a,Q)=1$. Then,
    $$ \frac{1}{x} \sum_{n \le x} f(Qn+a) = o(1). $$
    In particular,
    $$ \frac{1}{\log x} \sum_{n \le x} \frac{f(Qn+a)}{n} = o(1). $$
\end{lemma}
\begin{proof}
    See e.g. \cite[Theorem 2.3.4]{granvillemultiplicative} for the first estimate. Apply partial summation to obtain the second estimate.
\end{proof}

We now focus on non-asymptotic tools in pretentious theory. The idea is that a pretentious multiplicative function $f$ may still be at large distance from all functions of the form $\chi n^{it}$, where $\chi$ has bounded period and $t$ is bounded. In this case, averages of $f$ are small, although they do not necessarily tend to $0$.

\begin{definition}
    Let $x \ge A \ge 1$, say that $f$ is \textit{$\left(A,x\right)$-non-pretentious} if for all Dirichlet characters $\chi$ of period $\le A$ and for all $\left|t\right|\le Ax$, we have $\mathbb{D}(f,\chi n^{it};x)\geq A$.
\end{definition}

The next lemma, due to Tao \cite[Theorem 1.3]{tao2016logarithmically}, is the key result that gives us access to two-point correlations. It can be viewed as a non-asymptotic version of Lemma~\ref{lem:loghalasz} for two-point correlations.

\begin{lemma}\label{lem:logtao}
    Let $Q_1,Q_2,a_1,a_2 \in \mathbb{N}$ be such that $a_2Q_1-a_1Q_2 \neq 0$, and  $\varepsilon>0$. Then, there exists $A_0=A_0(Q_1,Q_2,a_1,a_2,\varepsilon) \ge 2$ such that for all $x \ge A \ge A_0$ and for all multiplicative functions $f,g:\mathbb{N}\to\mathbb{D}$ such that $g$ is $(A,x)$-non-pretentious, we have
    $$ \left| \frac{1}{\log x} \sum_{n \le x} \frac{f(Q_1n+a_1)g(Q_2n+a_2)}{n} \right| \le \varepsilon. $$
\end{lemma}

In the next Lemma, we prove that if $f$ is a multiplicative function such that a power of $f$ is pretentious, then all the other powers of $f$ that do not divide $k$ are non-pretentious in a non-asymptotic sense.

\begin{lemma}\label{lem:nonpret}
    Let $f:\mathbb{N}\rightarrow\mathbb{T}$ be a multiplicative function such that there exist a completely multiplicative function $h$, a minimal integer $k\geq1$, a Dirichlet character $\chi$ mod $q$, and a real number $t\in\mathbb{R}$ satisfying $h^k=\widetilde{\chi}$ and $\mathbb{D}\left(f,hn^{it};\infty\right)<\infty$. Let $m\in\mathbb{N}$ be such that $k\nmid m$. Then, there exists $A_0\geq1$ such that for every $A\geq A_0$, there exists $x_0\geq A$ such that for all $x\geq x_0$, the function $f^m$ is $\left(A,x\right)$-non-pretentious. Moreover, $A_0$ and $x_0$ may be chosen independently of $m$.
\end{lemma}
\begin{proof}
    Let us first show that it suffices to prove the lemma in the cases $f=h^c$, for every $1\le c\le k-1$. Indeed, assume that the lemma is true in these cases, and assume for the sake of contradiction that $f^m$ is not $\left(A,x\right)$-non-pretentious for some arbitrarily large $A$ and some arbitrarily large $x_A\geq\left(m+2\right)A$. Then, there exist a character $\chi^\prime$ of period $\le A$ and $\left|t^\prime\right|\le Ax_A$ such that $\mathbb{D}\left(f^m,\chi^\prime n^{it^\prime};x_A\right)<A$. Writing $m=kd+c$, where $1\le c\le k-1$, the second triangle inequality gives
    \begin{equation*}
        \begin{aligned}
            \mathbb{D}\left(f^c,\chi^\prime n^{it^\prime}\left(\chi n^{itk}\right)^{-d};x_A\right)
            &\le \mathbb{D}\left(f^m,\chi^\prime n^{it^\prime};x_A\right)+\mathbb{D}\left(f^{-kd},\left(\chi n^{itk}\right)^{-d};x_A\right) \\
            &< (m+1)A,
        \end{aligned}
    \end{equation*}
    since by another application of the second triangle inequality, we have
    $$ \mathbb{D}\left(f^{-kd},\left(\chi n^{itk}\right)^{-d};x_A\right)=\mathbb{D}\left(f^{kd},\left(\chi n^{itk}\right)^d;x_A\right)\le kd\mathbb{D}\left(f,hn^{it};x_A\right)\le mA, $$
    where we took $A \ge k \mathbb{D}\left(f,hn^{it};\infty\right)$. Noting that $\mathbb{D}\left(f^c,h^cn^{itc};x_A\right)\le A$, we deduce by the triangle inequality that
    $$ \mathbb{D}\left(h^cn^{itc},\chi^\prime n^{it^\prime}\left(\chi n^{itk}\right)^{-d};x_A\right)<\left(m+2\right)A. $$
    Note that $\chi^\prime\chi^{-d}$ is a Dirichlet character of period $\le\left(m+2\right)A$ and $\left|t^\prime-tm\right|\le\left(m+2\right)Ax_A$ by taking $A \ge |t|$. Since the above inequalities hold for all sufficiently large $A$ and $x_A$, this contradicts the lemma for $h^c$.

    From now on, fix $1\le c\le k-1$ and consider the case $f=h^c$. Let $A \ge 1$ be sufficiently large, $x\geq A$ sufficiently large, $\chi^\prime$ a Dirichlet character of period $3 \le q^\prime\le A$ and $\left|t^\prime\right|\le Ax$. Note that $h^c$ takes values in $\mu_K$, where $K := k\varphi(q)$.
    
    First, suppose that $\left|t^\prime\right|>\frac{3}{K}$. The second triangle inequality gives
    \begin{equation}\label{lem:nonpret-1}
        \mathbb{D}\left(h^c,\chi^\prime n^{it^\prime};x\right)\geq\frac{1}{K}\mathbb{D}\left(1,\left(\chi^\prime\right)^Kn^{it^\prime K};x\right).
    \end{equation}
    Now, the Vinogradov-Korobov zero-free region for Dirichlet $L$-functions (see \cite[Lemma 5.1]{khale2024explicit}) yields
    $$ \left|L\left(1+\frac{1}{\log{x}}+it^\prime K,\overline{\chi^\prime}^K\right)\right|\ll\log^{2/3}{\left|t^\prime K\right|}+\log{A}. $$
    Taking $x$ sufficiently large, it follows that
    \begin{equation*}
        \begin{aligned}
            \mathbb{D}^2\left(1,\left(\chi^\prime\right)^Kn^{it^\prime K};x\right)
            &\geq \log{\log{x}}-\log\left|L\left(1+\frac{1}{\log{x}}+it^\prime K,\overline{\chi^\prime}^K\right)\right|-O(1) \\
            &\geq \frac{1}{4}\log{\log{x}}-O(1),
        \end{aligned}
    \end{equation*}
    hence by \eqref{lem:nonpret-1}, we have
    $$ \mathbb{D}\left(h^c,\chi^\prime n^{it^\prime};x\right)\gg\sqrt{\log{\log{x}}}. $$
    
    Suppose now that $\left|t^\prime\right|\le\frac{3}{K}$. If
    \begin{equation}\label{lem:nonpret-2}
        \mathbb{D}\left(1,\left(\chi^\prime\right)^Kn^{it^\prime K};x\right)\geq\frac{1}{C}\mathbb{D}\left(h^c,1;x\right)
    \end{equation}
    for some fixed constant $C>0$, then we apply \eqref{lem:nonpret-1} to obtain
    $$ \mathbb{D}\left(h^c,\chi^\prime n^{it^\prime};x\right)\gg\mathbb{D}\left(h^c,1;x\right). $$
    Suppose that \eqref{lem:nonpret-2} does not hold. Assume furthermore that $\chi^\prime$ is non-principal. By \cite[Lemma 3.1.1]{granvillemultiplicative}, we have
    $$ \mathbb{D}^2\left(1,\left(\chi^\prime\right)^Kn^{it^\prime K};x\right)\asymp\log{\log{x}}\asymp\mathbb{D}^2\left(1,\chi^\prime n^{it^\prime};x\right). $$
    As a consequence,
    \begin{equation*}
        \begin{aligned}
            \mathbb{D}\left(h^c,1;x\right)
            &\le \mathbb{D}\left(h^c,\chi^\prime n^{it^\prime};x\right)+\mathbb{D}\left(1,\chi^\prime n^{it^\prime};x\right) \\
            &\ll \mathbb{D}\left(h^c,\chi^\prime n^{it^\prime};x\right)+\mathbb{D}\left(1,\left(\chi^\prime\right)^Kn^{it^\prime K};x\right),
        \end{aligned}
    \end{equation*}
    hence taking $C$ sufficiently large, we obtain by \eqref{lem:nonpret-2},
    $$ \mathbb{D}\left(h^c,\chi^\prime n^{it^\prime};x\right)\gg\mathbb{D}\left(h^c,1;x\right). $$
    Finally, in all cases, we have
    $$ \mathbb{D}\left(h^c,\chi^\prime n^{it^\prime};x\right)+\log{\log{\log{A}}}\gg\mathbb{D}\left(h^c,1;x\right), $$
    where if $\chi^\prime$ is principal, this is a consequence of \cite[Lemma 3.1]{klurman2018orbits} and
    $$ \sum_{p|q^\prime}\frac{1}{p}=O\left(\log{\log{\log{A}}}\right). $$
    It remains to prove that $\mathbb{D}\left(h^c,1;x\right)\geq C^\prime A$ for some fixed large constant $C^\prime$, and for some large $x$. Since $x$ can be chosen sufficiently large depending on $A$, it suffices to show that $\mathbb{D}\left(h^c,1;\infty\right)=\infty$. Indeed, if it were false, this would contradict the minimality of $k$.
\end{proof}

Finally, we state a concentration inequality for pretentious functions. Roughly speaking, it asserts that whenever $f$ pretends to $\chi n^{it}$, the values of $f(n)$ are, on average, well approximated by $\chi n^{it}$, provided that $n$ is restricted to a suitable arithmetic progression. The need for such a restriction arises from the fact that we view this concentration inequality as a multiplicative analogue of the Tur\'an--Kubilius inequality, which asserts that, on average,
$f$ is close to its mean over a prescribed set of integers. Since the mean value of
$f$ is expected to be close to that of $\chi n^{it}$, it is convenient, in order to retain the information carried by $\chi$, to restrict the summation to arithmetic progressions of the form $n = Qm + a$, where the modulus of $\chi$ divides $Q$.

\begin{lemma}\label{prop:concentration}
    Let $f:\mathbb{N}\to\mathbb{T}$ be a multiplicative function such that there exist a Dirichlet character $\chi$ mod $q$ and a real number $t$ such that $\mathbb{D}(f,\chi n^{it};\infty)<\infty$. Let $y \ge 1$, and $1 \le a \le Q$ such that $(Q,a)=1$ and $q,\prod_{p \le y} p|Q$, then for $Y \ge 3$ sufficiently large depending on the other parameters, we have
    $$
    \frac{1}{\log Y} \sum_{n \le Y} \frac{\left| f(Qn+a) - \chi(a)(Qn)^{it}\exp{(F(Q,Y))} \right|}{n}
    \ll \mathbb{D}(f,\chi n^{it};y,Y) + \frac{1}{\sqrt{y}},
    $$
    where
    $$ F(Q,Y) := \sum_{\substack{p \le Y \\ p \nmid Q}} \frac{f(p)\overline{\chi(p)}p^{-it}-1}{p}. $$
\end{lemma}
\begin{proof}
    This follows from \cite[Lemma 2.5]{klurman2021multiplicative}.
\end{proof}

\section{Technical lemmas on the distribution of multiplicative functions on the unit circle}\label{sec:lemmas}

In the first lemma of this section, we prove an equidistribution result for non-pseudo-pretentious multiplicative functions on an arithmetic progressions.

\begin{lemma}\label{equidistnonpret}
    Let $f$ be a non-pseudo-pretentious multiplicative function taking values in $\mathbb{T}$, and let $I\subset\mathbb{T}$ be an arc. Then, for all $1 \le a \le Q$ such that $\left(a,Q\right)=1$, we have
    $$ \frac{1}{\log x}\sum_{\substack{n \le x \\ f(Qn+a) \in I}} \frac{1}{n} \xrightarrow[x \to +\infty]{} \frac{\left|I\right|}{2\pi}. $$
\end{lemma}
\begin{proof}
    By Weyl's criterion and Hal\'asz's theorem in arithmetic progressions (Lemma~\ref{lem:loghalasz}), we have
    $$ \frac{1}{x} \#\{n \le x : f(Qn+a) \in I\} \xrightarrow[x \to +\infty]{} \frac{\left|I\right|}{2\pi}, $$
    hence the lemma follows by partial summation.
\end{proof}

In contrast with what happens for non-pseudo-pretentious functions, partial sums of the multiplicative function $n^{it}$ does not exhibit cancellation. Thus, we cannot apply Weyl's criterion. However, one may take advantage of the structure of $n^{it}$, that allows us to access the position of $n^{it}$ on the unit circle from the size of $n$.

\begin{lemma}\label{lem:logsinglenit}
    Let $t$ be a nonzero real number, $\delta\in\left(0,\min{\left(\left|t\right|,\frac{1}{3}\right)}\right)$, $z\in\mathbb{T}$ and $Y\geq Y_0$ for some $Y_0=Y_0\left(t,\delta\right)\geq2$. Then,
    $$ \sum_{\substack{n \le Y \\ n^{it}\in B_\delta(z)}} \frac{1}{n} \asymp_t\delta\log{Y}. $$
\end{lemma}
\begin{proof}
    Assume, without loss of generality, that $t>0$. Set $E \vcentcolon= e^{2\pi/t}$. Let $C$ be a large constant depending only on $t$ and $\delta$, write by an $E$-adic decomposition,
    $$ \sum_{\substack{n \le Y \\ n^{it}\in B_\delta(z)}} \frac{1}{n} \asymp_t \sum_{\frac{\log{C}}{\log{E}}<k\le\frac{\log{Y}}{\log{E}}} \frac{1}{E^k} \sum_{\substack{E^k<n\le E^{k+1}}} \mathbb{1}_{B_\delta(z)}\left(e\left(\frac{\log{n}}{\log{E}}\right)\right). $$
    If $n\in\left(E^k,E^{k+1}\right]$, then we have $\frac{\log{(n+1)}}{\log{E}}-\frac{\log{n}}{\log{E}}\le\frac{1}{n\log{E}}\le\frac{t}{2\pi C}<2\delta$ and $\frac{\log{(n+1)}}{\log{E}}-\frac{\log{n}}{\log{E}} \asymp_t \frac{1}{E^k}$, hence the set of $n \in \left(E^k,E^{k+1}\right]$ such that $\mathbb{1}_{B_\delta(z)}\left(e\left(\frac{\log{n}}{\log{E}}\right)\right)=1$ is an interval $I_k\subset\left(E^k,E^{k+1}\right]$ of length $\asymp_t\delta E^k$. We deduce that
    \begin{equation*}
        \begin{aligned}
            \sum_{\substack{n \le Y \\ n^{it}\in B_\delta(z)}} \frac{1}{n}
            \asymp_t \delta\sum_{\frac{\log{C}}{\log{E}}<k\le\frac{\log{Y}}{\log{E}}} 1
            \asymp_t \delta \log{Y},
        \end{aligned}
    \end{equation*}
    provided that $Y$ is sufficiently large in terms of $t$ and $\delta$.
\end{proof}

\begin{lemma}\label{lem:logdoublenit}
    Let $\left(t_1,t_2\right)\in\mathbb{R}^2$ be nonzero and linearly independent over $\mathbb{Q}$. For $\delta\in\left(0,\min{\left(\left|t_1\right|, \left|t_2\right|,\frac{1}{3}\right)}\right)$, $z, w\in\mathbb{T}$ and $Y\geq2$ sufficiently large depending on $t_1$, $t_2$ and $\delta$, we have
    $$ \sum_{\substack{n \le Y \\ \left(n^{it_1}, n^{it_2}\right)\in B_\delta(z)\times B_\delta(w)}} \frac{1}{n} \gg_{t_1,t_2}\delta^2\log{Y}. $$
\end{lemma}
\begin{proof}
    Taking conjugates of $z$ and $w$ if necessary, we can assume without loss of generality that $t_1,t_2>0$. Write $z=e\left(\theta_1\right)$ and $w=e\left(\theta_2\right)$ ($\left|\theta_1\right|, \left|\theta_2\right|\le\frac{1}{2}$), and define
    $$
    A(Y) \vcentcolon=
    \left\{ n\le Y :
    \begin{cases}
        \exists k\in\mathbb{Z},\ \left|\log{n}-\frac{2\pi\theta_1}{t_1}-\frac{2\pi k}{t_1}\right|<\frac{2\pi\delta}{t_1} \\
        \exists\ell\in\mathbb{Z},\ \left|\log{n}-\frac{2\pi\theta_2}{t_2}-\frac{2\pi\ell}{t_2}\right|<\frac{2\pi\delta}{t_2}
    \end{cases}
    \right\}.
    $$
    Then
    $$ \sum_{\substack{n \le Y \\ \left(n^{it_1}, n^{it_2}\right)\in B_\delta(z)\times B_\delta(w)}} \frac{1}{n} =\sum_{n\in A(Y)}\frac{1}{n}. $$
    Set $y \vcentcolon= \frac{t_2}{t_1}\in\mathbb{R}\backslash\mathbb{Q}$ and $x_0 \vcentcolon= \frac{t_2}{t_1}\theta_1 - \theta_2$, and define
    $$ B(Y) \vcentcolon= \left\{k\le\frac{t_1}{2\pi}(\log{Y}-\theta_1-\delta) : d\left(ky-x_0,\ \mathbb{Z}\right)<t_2\frac{\delta}{2\max{\left(t_1,t_2\right)}}\right\}. $$
    Then,
    \begin{equation*}
        \begin{aligned}
            \sum_{n\in A(Y)}\frac{1}{n}
            &\geq \sum_{k\in B(Y)}\sum_{\exp{\left(\frac{2\pi\theta_1}{t_1}+\frac{2\pi k}{t_1}-\frac{2\pi\delta}{2\max{\left(t_1,t_2\right)}}\right)}\le n\le\exp{\left(\frac{2\pi\theta_1}{t_1}+\frac{2\pi k}{t_1}+\frac{2\pi\delta}{2\max{\left(t_1,t_2\right)}}\right)}}\frac{1}{n} \\
            &\gg_{t_1,t_2} \delta\left|B(Y)\right|.
        \end{aligned}
    \end{equation*}
    Indeed, if $\left|\log n - \frac{2\pi\theta_1}{t_1}-\frac{2\pi k}{t_1}\right|<\frac{2\pi\delta}{2\max{(t_1,t_2)}}$ for some $k \in B(Y)$, then let $\ell \in \mathbb{Z}$ be such that $|ky-x_0-\ell|<t_2\frac{\delta}{2\max{\left(t_1,t_2\right)}}$, we have
    $$
    \left|\log{n}-\frac{2\pi\theta_2}{t_2}-\frac{2\pi\ell}{t_2}\right|
    \le \frac{2\pi\delta}{2\max{(t_1,t_2)}} + 2\pi\left| \frac{\theta_2}{t_2}-\frac{\theta_1}{t_1}+\frac{k}{t_1}-\frac{\ell}{t_2} \right|
    \le \frac{2\pi\delta}{t_2}.
    $$
    Next, write $B(Y)=\left\{k\le X : d\left(ky-x_0,\ \mathbb{Z}\right)<\eta\right\}$ for simplicity. It remains to prove that $B(Y)\gg\eta X$. This follows from the fact that $y$ is irrational, since an irrational rotation is equidistributed on the unit circle by Weyl’s criterion.
\end{proof}

\section{Reduction to the pretentious case}\label{sec:reduction}

Let $f$ and $g$ be as in Theorem~\ref{thm:main}. In what follows, all the implicit constants may depend on $z$, $w$, $f$ and $g$. Without loss of generality, we assume that $\varepsilon$ is sufficiently small depending on these quantities.

In the next proposition, we exclude the case when exactly one of the functions $f$ and $g$ is pseudo-pretentious.

\begin{proposition}\label{prop:only1pret}
    The functions $f$ and $g$ are either both non-pseudo-pretentious or both pseudo-pretentious.
\end{proposition}
\begin{proof}
    Assume, for the sake of contradiction, and without loss of generality, that $g$ is non-pseudo-pretentious and that there exists a minimal positive integer $k$, a completely multiplicative function $h:\mathbb{N}\to\mathbb{T}$ such that $h^k=\widetilde{\chi}$, where $\chi$ is a Dirichlet character mod $q$, and a real number $t$, such that $\mathbb{D}\left(f,hn^{it};\infty\right)<\infty$.

    Let $1\le y\le Y$ with $y$ (resp. $Y$) sufficiently large in terms of $\varepsilon^{-1}$ (resp. $y$). The dependence between the parameters may be summarized as follows:
    $$ z,w,f,g,h,k,t,\chi,q \lll \varepsilon^{-1} \lll y \lll Y. $$
    
    Set $W_1 \vcentcolon= 2q\prod_{p\le y} p$, and take an integer $W_2 \ge 1$ such that $P^+(W_2) \le y$, which depends at most on $y$ and the underlying quantities. We take $y$ such that $P^+(q)\le y$. For $\ell\geq1,$ define $n_\ell=W^\prime\ell+n_0 \vcentcolon= W_2W_1(W_1\ell+1)$.
    
    Take $y$ and $Y$ sufficiently large, and apply Lemma~\ref{prop:concentration} to the pretentious multiplicative function $f\overline{h}$ to obtain
    \begin{equation}\label{prop:only1pret-1}
        \begin{multlined}
            \frac{1}{\log Y} \sum_{\substack{\ell \le Y}} \frac{\left|f(W_1\ell+1)-h(W_1\ell+1)\left(W_1\ell\right)^{it}\exp{\left(iI(y,Y)+O\left(\varepsilon^7\right)\right)}\right|}{\ell}
            \le\varepsilon^6,
        \end{multlined}
    \end{equation}
    where
    $$ I(y,Y) \vcentcolon= \sum_{y<p\le Y}{\frac{\mathfrak{I}{\left(f(p)\overline{h(p)}p^{-it}\right)}}{p}}. $$
    By a compactness argument, restricting $Y$ to belong to an increasing sequence $\left(Y_j\right)_{j\geq1}$ (which depends on $y$) such that $Y_j \to +\infty$, we can assume that $\exp{\left(iI(y,Y_j)\right)}\rightarrow\gamma$ for some $\gamma\in\mathbb{T}$ depending on $y$. Then, for $j$ sufficiently large in terms of $y$ and the underlying quantities, we have $\exp{\left(iI(y,Y_j)\right)}=\gamma+O\left(\varepsilon^7\right)$, hence \eqref{prop:only1pret-1} yields
    $$ \frac{1}{\log Y_j} \sum_{\substack{\ell \le Y_j}} \frac{\left|f(W_1\ell+1)-h(W_1\ell+1)\left(W_1\ell\right)^{it}\gamma\right|}{\ell} \le2\varepsilon^6. $$
    Now, define
    $$
    \mathcal{I}_j \vcentcolon=
    \left\{\ell\le Y_j : \left|f(W_1\ell+1)-h(W_1\ell+1)\left(W_1\ell\right)^{it}\gamma\right|\le2\varepsilon^3\right\}.
    $$
    Then, by Rankin’s trick, we have
    \begin{equation}\label{prop:only1pret-2}
        \begin{aligned}
            \frac{1}{\log Y_j} \sum_{\substack{\ell \le Y_j \\ \ell\notin\mathcal{I}_j}} \frac{1}{\ell}
            \le \frac{1}{2\varepsilon^3\log{Y_j}} \sum_{\substack{\ell \le Y_j}} \frac{\left|f(W_1\ell+1)-h(W_1\ell+1)\left(W_1\ell\right)^{it}\gamma\right|}{\ell}
            \le \varepsilon^3.
        \end{aligned}
    \end{equation}
    In addition, for all $\ell\in\mathcal{I}_j$, we have
    $$ f(n_\ell)=f(W_2)f(W_1)h(W_1\ell+1)W_1^{it}\gamma\ell^{it}+O\left(\varepsilon^3\right). $$
    By density of $\left\{f(n)\right\}_n$, let us choose $W_2$ such that  $f(W_2) = z\overline{f(W_1)W_1^{it}\gamma}+O(\varepsilon^3)$. This is possible despite the restriction $P^+(W_2) \le y$, since we can take $y$ sufficiently large so that there exists a prime $p \le y$ such that the set $\{f(p^n)\}_{n \ge 1}$ is at distance at most $\varepsilon^3$ from every point of $\mathbb{T}$. Thus, it suffices to fix such a $p$, take $W_2=p^n$, and choose $n$ in terms of $W_1$, $f$ and $t$. Therefore,
    \begin{equation}\label{prop:only1pret-3}
        f(n_\ell) = zh(W_1\ell+1)\ell^{it}+O\left(\varepsilon^3\right).
    \end{equation}
    Next, define
    $$
    \mathcal{I}_j^\prime \vcentcolon= 
    \left\{
    \ell \in \mathcal{I}_j :
    \begin{cases}
        n_\ell\notin\mathcal{A}\left(W'Y_j+n_0\right) \\
        \ell^{it}\notin B_{\varepsilon^3}\left(e\left(\pm\frac{\varepsilon}{10}\right)\right) \\
        g(n_\ell+1)\notin B_{\varepsilon^3}\left(we\left(\pm\frac{\varepsilon}{10}\right)\right)
    \end{cases}
    \right\}.
    $$
    By the union bound, Lemma~\ref{lem:logsinglenit} and \eqref{prop:only1pret-2}, we have
    \begin{equation*}
        \begin{aligned}
            \frac{1}{\log Y_j} \sum_{\substack{\ell \le Y_j \\ \ell\notin\mathcal{I}_j'}} \frac{1}{\ell}
            \le & \frac{1}{\log Y_j} \sum_{\substack{\ell \le Y_j  \\n_\ell\in\mathcal{A}\left(W'Y_j+n_0\right)}} \frac{1}{\ell} \\
            &+ \frac{1}{\log{Y_j}} \sum_{0\le r\le\log{Y_j}} \frac{1}{e^r} \sum_{\substack{e^r\le\ell\le e^{r+1}}} \mathbb{1}_{B_{\varepsilon^3}\left(we\left(\pm\frac{\varepsilon}{10}\right)\right)}\left(g(n_\ell+1)\right)
            + O\left(\varepsilon^3\right).
        \end{aligned}
    \end{equation*}
    Take $j$ sufficiently large in terms of $W'$ and $\varepsilon$. Using Lemma~\ref{equidistnonpret}, we deduce that
    \begin{equation}\label{prop:only1pret-4}
        \begin{aligned}
            \frac{1}{\log Y_j} \sum_{\substack{\ell \le Y_j \\ \ell\notin\mathcal{I}_j'}} \frac{1}{\ell}
            &\ll \frac{W'}{\log{Y_j}} \sum_{n\in\mathcal{A}\left(W'Y_j+n_0\right)} \frac{1}{n} \\
            &\quad + \frac{1}{\log{Y_j}} \sum_{0\le r\le\log{Y_j}} \frac{1}{e^r} \sum_{\substack{\ell\le e^{r+1}}} \mathbb{1}_{B_{\varepsilon^3}\left(we\left(\pm\frac{\varepsilon}{10}\right)\right)}\left(g(n_\ell+1)\right)
            + \varepsilon^3 \\
            &\ll \varepsilon^3,
        \end{aligned}
    \end{equation}
    since $d_{\log}(\mathcal{A})=0$.

    Set
    $$ S_j \vcentcolon= \frac{1}{\log{Y_j}}\sum_{\ell\in\mathcal{I}_j^\prime}\frac{\mathbb{1}_{B_\frac{\varepsilon}{10}(1)}\left(\ell^{it}\right) \mathbb{1}_{h(W_1\ell+1)=1} \mathbb{1}_{B_\frac{\varepsilon}{10}(w)}\left(g(n_\ell+1)\right)}{\ell}. $$
    Recalling \eqref{prop:only1pret-3} and the definition of $\mathcal{A}\left(W'Y_j+n_0\right)$, since $\varepsilon$ is assumed to be sufficiently small, note that $S_j=0$. Now, by definition of $\mathcal{I}_j^\prime$ and applying Lemma~\ref{lem:fourierapprox}, as well as the identity
    $$ \mathbb{1}_{h(W_1\ell+1)=1}=\frac{1}{k\varphi(q)}\sum_{b=0}^{k\varphi(q)-1}{h^b(W_1\ell+1)}, $$
    which follows from the fact that $h(W_1\ell+1)\in\mu_{k\varphi(q)}$, we obtain for $X\geq1$,
    \begin{equation}\label{prop:only1pret-4.5}
        \begin{multlined}
            S_j
            = \sum_{\left|a\right|,|c|\le X}\sum_{b=0}^{k\varphi(q)-1}\frac{\sin{\left(2\pi a\frac{\varepsilon}{10}\right)}\sin{\left(2\pi c\frac{\varepsilon}{10}\right)}}{\pi^2ac  w^c k\varphi(q)} \frac{1}{\log{Y_j}}\sum_{\ell\in\mathcal{I}_j^\prime}\frac{\ell^{ita}h^b(W_1\ell+1)g^c(n_\ell+1)}{\ell} \\
            + O\left( \frac{1}{\varepsilon^3X} \right).
        \end{multlined}
    \end{equation}
    Assume that $X \le W_1$. Writing $\ell^{ita} = \left(\frac{W_1\ell+1}{W_1}\right)^{ita} + O\left(\frac{X}{W_1}\right)$, and then applying \eqref{prop:only1pret-4}, we deduce that
    \begin{equation*}
        \begin{multlined}
            S_j
             = \sum_{\left|a\right|,|c|\le X}\sum_{b=0}^{k\varphi(q)-1}\frac{\sin{\left(2\pi a\frac{\varepsilon}{10}\right)}\sin{\left(2\pi c\frac{\varepsilon}{10}\right)}}{\pi^2acW_1^{ita}  w^ck\varphi(q)} \\
             \qquad\qquad\qquad\qquad \times \frac{1}{\log{Y_j}}\sum_{\ell\le Y_j}\frac{(W_1\ell+1)^{ita}h^b(W_1\ell+1)g^c(n_\ell+1)}{\ell} \\
             + O\left(\frac{1}{\varepsilon^3X}+\left(\varepsilon^3+\frac{X}{W_1}\right)\left(\log{X}\right)^2\right),
        \end{multlined}
    \end{equation*}
    Let $\delta>0$ to be chosen later. If $k\nmid b$, apply Lemma~\ref{lem:nonpret} to $h^b$. Take $j$ sufficiently large in terms of $y$ and the underlying quantities, and in terms of $\delta$ and $X$. Then, Lemma~\ref{lem:logtao} implies that for all integers $b \in [0,k\varphi(q)-1]$ such that $k \nmid b$, and for all $|c| \le X$, we have
    $$
    \left| \frac{1}{\log{Y_j}}\sum_{\ell\le Y_j}\frac{(W_1\ell+1)^{ita}h^b(W_1\ell+1)g^c(n_\ell+1)}{\ell} \right|
    \le \delta.
    $$
    Thus,
    \begin{equation*}
        \begin{multlined}
            S_j
            = \sum_{\left|a\right|,|c|\le X}\sum_{d=0}^{\varphi(q)-1}\frac{\sin{\left(2\pi a\frac{\varepsilon}{10}\right)}\sin{\left(2\pi c\frac{\varepsilon}{10}\right)}}{\pi^2acW_1^{ita}  w^c k\varphi(q)} \\
            \qquad\qquad\qquad\qquad \times \frac{1}{\log{Y_j}}\sum_{\ell\le Y_j}\frac{(W_1\ell+1)^{ita}\chi^d(W_1\ell+1)g^c(n_\ell+1)}{\ell} \\
            + O\left(\frac{1}{\varepsilon^3X}+\left(\delta+\varepsilon^3+\frac{X}{W_1}\right)\left(\log{X}\right)^2\right).
        \end{multlined}
    \end{equation*}
    Then, since $\chi(W_1\ell+1)=1$, we obtain
    \begin{equation*}
        \begin{multlined}
            S_j
            = \sum_{\left|a\right|,|c|\le X}{\frac{\sin{\left(2\pi a\frac{\varepsilon}{10}\right)}\sin{\left(2\pi c\frac{\varepsilon}{10}\right)}}{\pi^2ac\left(W_2W_1^2\right)^{ita}  w^ck}\frac{1}{\log{Y_j}}\sum_{\ell\le Y_j}\frac{(n_\ell+1)^{ita}g^c(n_\ell+1)}{\ell}} \\
            + O\left(\frac{1}{\varepsilon^3X}+\left(\delta+\varepsilon^3+\frac{X}{W_1}\right)\left(\log{X}\right)^2\right).
        \end{multlined}
    \end{equation*}
    Now, $g^c$ is non-pretentious for $c\neq0$, hence by Lemma~\ref{lem:loghalasz}, taking $j$ sufficiently large in terms of $X$, as well as $y$ and the underlying quantities, we have
    \begin{equation*}
        \begin{multlined}
            S_j
            = \frac{\varepsilon}{5k}\sum_{\left|a\right|\le X}{\frac{\sin{\left(2\pi a\frac{\varepsilon}{10}\right)}}{\pi a\left(W_2W_1^2\right)^{ita}}\frac{1}{\log{Y_j}}\sum_{\ell\le Y_j}\frac{(n_\ell+1)^{ita}}{\ell}} \\
            + O\left(\frac{1}{\varepsilon^3X}+\left(\delta+\varepsilon^3+\frac{X}{W_1}\right)\left(\log{X}\right)^2\right).
        \end{multlined}
    \end{equation*}
    We use \eqref{prop:only1pret-4} and note that $\left(\frac{n_\ell+1}{W_2W_1^2}\right)^{ita} = \ell^{it} + O\left(\frac{X}{W_1}\right) $ to deduce that
    \begin{equation*}
        \begin{multlined}
            S_j
            = \frac{\varepsilon}{5k}\frac{1}{\log{Y_j}}\sum_{\ell\in\mathcal{I}_j^\prime}{\frac{1}{\ell}\sum_{\left|a\right|\le X}{\frac{\sin{\left(2\pi a\frac{\varepsilon}{10}\right)}}{\pi a }\ell^{ita}}}
            + O\left(\frac{1}{\varepsilon^3X}+\left(\delta+\varepsilon^3+\frac{X}{W_1}\right)\left(\log{X}\right)^2\right).
        \end{multlined}
    \end{equation*}
    Then, apply Lemma~\ref{lem:fourierapprox} to obtain
    \begin{equation*}
        \begin{aligned}
            S_j
            &= \frac{\varepsilon}{5k}\frac{1}{\log{Y_j}}\sum_{\ell\le Y_j}\frac{\mathbb{1}_{B_\frac{\varepsilon}{10}(1)}\left(\ell^{it}\right)}{\ell}
            + O\left(\frac{1}{\varepsilon^3X}+\left(\delta+\varepsilon^3+\frac{X}{W_1}\right)\left(\log{X}\right)^2\right).
        \end{aligned}
    \end{equation*}
    Choose $X=\varepsilon^{-6}$ and $\delta=\varepsilon^3$, and take $y$ sufficiently large in terms of $\varepsilon$. Then, by Lemma~\ref{lem:logsinglenit} (the inequality is trivial when $t=0$), we have
    $$ S_j\gg\varepsilon^2, $$
    which contradicts $S_j=0$.
\end{proof}

In order to prove that $f$ and $g$ are pseudo-pretentious, we will use the following result of Klurman and Mangerel \cite[Proposition 2.2]{klurman2018orbits}.

\begin{lemma}\label{lem:impliespret}
    Let $f,g:\mathbb{N}\to\mathbb{T}$ be multiplicative functions. Assume that there are positive integers $a$ and $b$ such that
    $$ \frac{1}{\log x} \sum_{n \le x} \frac{f^a(n)g^b(n+1)}{n} \gg 1. $$
    Then, $f$ and $g$ are pseudo-pretentious.
\end{lemma}

We now reduce the problem to the pretentious case.

\begin{theorem}\label{thm:redpret}
    The functions $f$ and $g$ are pseudo-pretentious.
\end{theorem}
\begin{proof}
    Suppose for contradiction that at least one of $f$ and $g$ is non-pseudo-pretentious. By Proposition~\ref{prop:only1pret}, both are non-pseudo-pretentious. Let $x\ge1$, define
    $$ S \vcentcolon= \frac{1}{\log x} \sum_{\substack{n \le x \\ n\notin\mathcal{A}(x)}} \frac{\mathbb{1}_{B_\frac{\varepsilon}{10}(z)}\left(f(n)\right)\mathbb{1}_{B_\frac{\varepsilon}{10}(w)}\left(g(n+1)\right)}{n}. $$
    Assume that $x$ is sufficiently large so that
    $$ \frac{1}{\log x}\sum_{n\in\mathcal{A}(x)}\frac{1}{n}\le\varepsilon^3. $$
    By definition of $\mathcal{A}(x)$, we have $S=0$. Moreover, by Lemma~\ref{lem:fourierapprox}, for $X\geq1$, we have
    \begin{equation*}
        \begin{aligned}
            S
            &= \frac{1}{\log{x}}\sum_{\left|a\right|,\left|b\right|\le X}{\frac{\sin{\left(2\pi a\frac{\varepsilon}{10}\right)}\sin{\left(2\pi b\frac{\varepsilon}{10}\right)}}{\pi^2abz^aw^b}\sum_{n\notin\mathcal{A}(x)}\frac{f^a(n)g^b(n+1)}{n}}+O\left(\frac{1}{\varepsilon^2X}\right) \\
            &= \frac{\varepsilon^2}{25}+\frac{1}{\log{x}} \sum_{\substack{\left|a\right|,\left|b\right|\le X \\ (a,b)\neq(0,0)}} \frac{\sin{\left(2\pi a\frac{\varepsilon}{10}\right)}\sin{\left(2\pi b\frac{\varepsilon}{10}\right)}}{\pi^2abz^aw^b}\sum_{n\le x}\frac{f^a(n)g^b(n+1)}{n} \\
            &\quad + O\left(\frac{1}{\varepsilon^2X} + \varepsilon^3\left(\log{X}\right)^2 + \frac{\varepsilon^2}{\log x}\right).
        \end{aligned}
    \end{equation*}
    Take $X=\varepsilon^{-5}$ and $x \ge e^{100}$. By the triangle inequality, we have
    $$ \frac{\varepsilon^2}{50}\pi^2\log{x} \le \sum_{\substack{\left|a\right|,\left|b\right|\le X \\ (a,b)\neq(0,0)}} \frac{1}{ab}\left|\sum_{n\le x}\frac{f^a(n)g^b(n+1)}{n}\right|. $$
    By the pigeonhole principle, there exist $a,b\in\mathbb{Z}$ (independent of $x$) with $\left|a\right|,\left|b\right|\le X$ and $\left(a,b\right)\neq0$ such that
    \begin{equation}\label{thm:redpret-1}
        \left|\sum_{n\le x}\frac{f^a(n)g^b(n+1)}{n}\right|\gg_\varepsilon\log{x}
    \end{equation}
    for all $x \in \mathcal{X}$, where $\mathcal{X}$ is a set of positive integers whose lower asymptotic density is at least $\frac{1}{10X^2}$. Then, note that for all $x \in \mathcal{X}$ sufficiently large depending on $X$, there is an element of $\mathcal{X}$ in the interval $(x,11X^2x]$. Thus, \eqref{thm:redpret-1} holds for all $x$ sufficiently large depending on $\varepsilon$. Let us deduce that $f$ and $g$ are pseudo-pretentious. Since pseudo-pretentiousness is invariant under conjugation, we may assume without loss of generality that $a,b\geq0$. Furthermore, we must have $a,b\geq1$, since otherwise the above inequality would contradict Lemma~\ref{lem:loghalasz}. We then conclude by Lemma~\ref{lem:impliespret}.
\end{proof}

\section{Proof of the main theorem}\label{sec:proof}

Let $f$ and $g$ be as in Theorem~\ref{thm:main}. We know by Theorem~\ref{thm:redpret} that there are completely multiplicative functions $h_1$ and $h_2$, integers $k_1,k_2\geq1$, Dirichlet characters $\chi_1$ and $\chi_2$ of respective periods $q_1$ and $q_2$, and $t_1,t_2\in\mathbb{R}$ such that $h_1^{k_1}=\widetilde{\chi_1}$, $h_2^{k_2}=\widetilde{\chi_2}$ and
$\mathbb{D}\left(f,h_1n^{it_1};\infty\right), \mathbb{D}\left(g,h_2n^{it_2};\infty\right)<\infty$.
We take $k_1$ and $k_2$ minimal.

In this section, we will introduce the parameters $y$ and $Y$. Furthermore, without loss of generality, $\varepsilon$ may be supposed to be sufficiently small. Using a similar convention to that in the proof of Proposition~\ref{prop:only1pret}, some parameters are assumed to be large relatively to other parameters, following the hierarchy
$$ z,w,f,g,h_1,h_2,k_1,k_2,t_1,t_2,\chi_1,\chi_2,q_1,q_2 \lll \varepsilon^{-1} \lll y \lll Y. $$

In the next two propositions, we construct arithmetic progressions on which $f^{k_1}$ and $g^{k_2}$ take specific values. We make two different constructions depending on whether $(t_1,t_2)=(0,0)$ or not.

\begin{proposition}\label{prop:proofA}
    Assume that $\left(t_1,t_2\right)\neq\left(0,0\right)$ and that the conclusion of Theorem~\ref{thm:main} does not hold. Then, there exist
    \begin{itemize}
        \item   a real number $y \ge 1$ sufficiently large in terms of $\varepsilon$ and the underlying quantities;
        \item   positive integers $n_0 \le W^\prime$ only depending on $y$ and the underlying quantities;
        \item   a subset $\mathcal{I}_j\subset [1,Y_j]\cap\mathbb{N}$ such that $\frac{1}{\log{Y_j}}\sum_{\ell\in\mathcal{I}_j}\frac{1}{\ell}\geq1-O\left(\varepsilon^3\right)$ for all $j$ sufficiently large in terms of $y$ and the underlying quantities;
        \item complex numbers $u,v\in\mathbb{T}$ independent of $j$,
    \end{itemize}
    such that, if $n_\ell \vcentcolon= W^\prime\ell+n_0$, then
    $$
    \forall\ell\in\mathcal{I}_j,
    \begin{cases}
        f^{k_1}(n_\ell)=u^{k_1}\ell^{it_1k_1}+O\left(\varepsilon^9\right), \\
        g^{k_2}(n_\ell+1)=v^{k_2}\ell^{it_2k_2}+O\left(\varepsilon^9\right).
    \end{cases}
    $$
    In addition, if $t_1=0$ (resp. $t_2=0$), then we can take $u=z$ (resp. $v=w$). If there are coprime $s_1,s_2 \in \mathbb{Z}$ such that $s_1 t_1 = s_2 t_2 \neq 0$, then we can require $\left(\frac{z}{u}\right)^{s_1}=\left(\frac{w}{v}\right)^{s_2}+O(\varepsilon^9)$.
\end{proposition}
\begin{proof}
    If at least one of $t_1$ and $t_2$ is nonzero, then we assume without loss of generality that $t_2\neq0$.

    Let $y\geq 1$ be sufficiently large in terms of $\varepsilon$ and the underlying quantities, and $\left(\alpha_p\right)_{p\le y}$ be parameters, where the $\alpha_p$ are positive integers. We assume that the $\alpha_p$ depend at most on $y$ and the underlying quantities. Set $W \vcentcolon= q_1q_2\prod_{p\le y} p^{\alpha_p}$. Let $r\le W$ be such that $r\equiv1\left[q_1q_2\right]$ and $P^-(r)>y$ (such an $r$ exists by the Chinese remainder theorem, assuming for example that $p|r+1$ for all $p\le y$ such that $p\nmid q_1q_2$). For $\ell\geq1$, set $n_\ell=W^\prime\ell+n_0 \vcentcolon= W(2W\ell+r)$.

    Since $f^{k_1}$ is pretentious (by the second triangle inequality), taking $y$ and $Y$ sufficiently large, Lemma~\ref{prop:concentration} yields
    $$ \frac{1}{\log{Y}}\sum_{\ell\le Y}{\frac{1}{\ell}\left|f^{k_1}(2W\ell+r)-\left(2W\ell\right)^{it_1k_1}\exp{\left(iI_1(y,Y)+O\left(\varepsilon^{13}\right)\right)}\right|}\le\varepsilon^{12} $$
    and
    $$ \frac{1}{\log{Y}}\sum_{\ell\le Y}{\frac{1}{\ell}\left|g^{k_2}\left(2W^2\ell+Wr+1\right)-\left(2W^2\ell\right)^{it_2k_2}\exp{\left(iI_2(y,Y)+O\left(\varepsilon^{13}\right)\right)}\right|}\le\varepsilon^{12}, $$
    where
    $$ I_1(y,Y) \vcentcolon= \sum_{y<p\le Y}{\frac{\mathfrak{I}{\left(f^{k_1}(p)\overline{\widetilde{\chi_1}(p)}p^{-it_1 k_1}\right)}}{p}} $$
    and
    $$ I_2(y,Y) \vcentcolon= \sum_{y<p\le Y}{\frac{\mathfrak{I}{\left(g^{k_2}(p)\overline{\widetilde{\chi_2}(p)}p^{-it_2 k_2}\right)}}{p}}. $$
    By a compactness argument, there is an increasing sequence $(Y_j)_{j\geq1}$ of large values of $Y$ such that
    $$ \left(e^{iI_1(y,Y_j)}, e^{iI_2(y,Y_j)}\right) \xrightarrow[j\to+\infty]{} \left(\gamma_1, \gamma_2\right)\in\mathbb{T}^2, $$
    for some $\gamma_1$ and $\gamma_2$ depending on $y$. Then, for all $j\geq1$ sufficiently large in terms of $y$ and the underlying quantities, we have
    $$ \frac{1}{\log{Y_j}}\sum_{\ell\le Y_j}{\frac{\left|f^{k_1}(2W\ell+r)-\left(2W\ell\right)^{it_1k_1}\gamma_1\right|}{\ell}}\le2\varepsilon^{12} $$
    and
    $$ \frac{1}{\log{Y_j}}\sum_{\ell\le Y_j}{\frac{\left|g^{k_2}\left(2W^2\ell+Wr+1\right)-\left(2W^2\ell\right)^{it_2k_2}\gamma_2\right|}{\ell}}\le2\varepsilon^{12}. $$
    
    Define
    $$
    \mathcal{I}_j \vcentcolon=
    \left\{
    \ell \le Y_j :
    \begin{cases}
        \left|f^{k_1}(2W\ell+r)-\left(2W\ell\right)^{it_1k_1}\gamma_1\right|\le 4\varepsilon^9 \\
        \left|g^{k_2}\left(2W^2\ell+Wr+1\right)-\left(2W^2\ell\right)^{it_2k_2}\gamma_2\right|\le 4\varepsilon^9
    \end{cases}
    \right\}.
    $$
    By the union bound and by Rankin’s trick, we have
    \begin{equation*}
        \begin{aligned}
            \sum_{\substack{\ell \le Y_j \\ \ell\notin\mathcal{I}_j}} \frac{1}{\ell}
            &\le \frac{1}{4\varepsilon^9} \sum_{\ell\le Y_j}{\frac{\left|f^{k_1}(2W\ell+r)-\left(2W\ell\right)^{it_1k_1}\gamma_1\right|}{\ell}} \\
            &\quad + \frac{1}{4\varepsilon^9} \sum_{\ell\le Y_j}{\frac{\left|g^{k_2}\left(2W^2\ell+Wr+1\right)-\left(2W^2\ell\right)^{it_2k_2}\gamma_2\right|}{\ell}} \\
            &\le \varepsilon^3\log{Y_j},
        \end{aligned}
    \end{equation*}
    hence
    \begin{equation*}
        \frac{1}{\log{Y_j}}\sum_{\ell\in\mathcal{I}_j}\frac{1}{\ell}\geq1-O\left(\varepsilon^3\right).
    \end{equation*}
    In addition,
    \begin{equation}\label{prop:proofA-1}
        \forall\ell\in\mathcal{I}_j,
        \begin{cases}
            f^{k_1}(n_\ell)=\left(2W\ell\right)^{it_1k_1}f^{k_1}(W)\gamma_1+O\left(\varepsilon^9\right), \\
            g^{k_2}(n_\ell+1)=\left(2W^2\ell\right)^{it_2k_2}\gamma_2+O\left(\varepsilon^9\right).
        \end{cases}
    \end{equation}

    If we have $s_1t_1=s_2t_2\neq0$ for some $s_1, s_2\in\mathbb{Z}$ coprime, then by assumption and by exchanging $f$ and $g$ if necessary, we can assume that there exists a prime $p$ such that $\frac{f(p)}{p^{it_1}}=e(\theta_p)$ for some irrational $\theta_p$. Thus, taking $y \ge p$, we can suitably choose $\alpha_p$ to have
    $$ \frac{f(W)}{W^{it_1}} = \left(\frac{z^{s_1}\gamma_2^{s_2/k_2}}{w^{s_2}\gamma_1^{s_1/k_1}}\right)^{1/s_1}+O\left(\varepsilon^9\right), $$
    where each root of unity may be chosen arbitrarily. Take $u = \left(2W\right)^{it_1}f(W)\gamma_1^{1/k_1}$ and $v = \left(2W^2\right)^{it_2}\gamma_2^{1/k_2}$, where each root of unity may be chosen arbitrarily. In particular, we have
    $$
    \left(\frac{z}{u}\right)^{s_1}\left(\frac{w}{v}\right)^{-s_2}
    = \frac{z^{s_1}\gamma_2^{s_2/k_2}}{w^{s_2}\gamma_1^{s_1/k_1}} \left(\frac{W^{it_1}}{f(W)}\right)^{s_1}
    = 1 + O(\varepsilon^9),
    $$
    hence $ \left(\frac{z}{u}\right)^{s_1} = \left(\frac{w}{v}\right)^{s_2} + O(\varepsilon^9) $.
    
    If $t_1=0$ and $t_2\neq0$, then by density, for $y$ sufficiently large in terms of $\varepsilon$ (it is irrelevant that $\gamma_1$ depends on $y$ since it suffices to have a prime $p \le y$ such that the set $\{f(p^\alpha)\}_{\alpha \in \mathbb{N}}$ is sufficiently close to every point of $\mathbb{T}$ in terms of $\varepsilon$), we can choose $W$ such that $f(W)=z\gamma_1^{-1/k_1}+O\left(\varepsilon^9\right)$.
    
    If $t_1, t_2\neq0$ and $\left(t_1,t_2\right)$ is linearly independent over $\mathbb{Q}$, then we choose $W$ arbitrarily (e.g. $\alpha_p=1$ for all $p\le y$).
    
    In all the above cases, by \eqref{prop:proofA-1} there exist $u,v\in\mathbb{T}$ independent of $\ell$ and $j$ such that
    $$
    \forall\ell\in\mathcal{I}_j,
    \begin{cases}
        f^{k_1}(n_\ell)=u^{k_1}\ell^{it_1k_1}+O\left(\varepsilon^9\right), \\
        g^{k_2}(n_\ell+1)=v^{k_2}\ell^{it_2k_2}+O\left(\varepsilon^9\right),
    \end{cases}
    $$
    with $u=z$ if $t_1=0$.
\end{proof}

\begin{definition}
    Let $f \in \overline{\mathcal{M}}$. We say that $f$ is \textit{eventually rational} if
    $$ \exists k\in\mathbb{N},\ \exists N_0\in\mathbb{N},\ \forall p\geq N_0,\ f^k(p)=1. $$
    Otherwise, we say that $f$ is $\textit{irrational}$.
\end{definition}

\begin{proposition}\label{prop:proofB}
    The conclusion of Proposition~\ref{prop:proofA} holds if one assumes that $t_1=t_2=0$.
\end{proposition}
\begin{proof}
    We first show that there exist distinct primes $p_1$ and $p_2$ such that the sets $\left\{f\left(p_1^n\right)\right\}_n$ and $\left\{g\left(p_2^n\right)\right\}_n$ are within distance at most $\varepsilon^9$ to every point of $\mathbb{T}$. If $f$ or $g$ is irrational, then this immediately follows by definition of irrationality and the fact that $\overline{\left\{f_j(n)\right\}_n}=\mathbb{T}$ ($j\in\{1,2\}$). Otherwise, $f$ and $g$ are both eventually rational. Thus, there exist $k,N_0 \in \mathbb{N}$ such that for all $p \ge N_0$, we have $f^k(p)=g^k(p)=1$. Since $\left\{f^k(n)\right\}_n$ and $\left\{g^k(n)\right\}_n$ are dense, there are primes $p_1,p_2 < N_0$ such that, if we write $f(p_1)=e(\alpha_1)$ and $g(p_2)=e(\alpha_2)$, then $\alpha_1,\alpha_2 \in \mathbb{R}\backslash\mathbb{Q}$. We can choose $p_1 \neq p_2$ thanks to the condition $(f^K,g^K) \notin \mathcal{F}$ for all positive integer $K$ such that $k|K$. By exchanging $f$ and $g$ if necessary, let us assume that $p_2 \neq 2$.

    Set $W \vcentcolon= q_1q_2\prod_{p\le y} p$, where $y\geq1$ is sufficiently large in terms of $\varepsilon$ and the underlying quantities. Take $y \ge p_1,p_2,P^+(q_1),P^+(q_2)$. Let $m_1,m_2 \ge 1$ be parameters depending at most on $\varepsilon$ and underlying quantities. For $\ell\geq1$, define $n_\ell = W'\ell+n_0 \vcentcolon= 2 p_1^{m_1}\left(2W p_2^{m_2}\ell+r\right)$ for some $r \ge 1$ independent of $\ell$. By the Chinese remainder theorem, we can take $r \le W p_2^{m_2}$ such that $p_2^{m_2}|2 p_1^{m_1}r+1$ and $P^-(r),P^-\left(\frac{2 p_1^{m_1}r+1}{p_2^{m_2}}\right)>y$ (consider e.g. the relations $2 p_1^{m_1}r+1 \equiv p_2^{m_2}(p_2-1) \ \left[ p_2^{m_2+1}\right]$, $r \equiv 1 \ [p_1]$, $r \equiv 1 \ [2]$, and for $3 \le p \le y$ such that $p \neq p_1,p_2$, $r \equiv r_p \ [p]$, where $-r_p \in \mathbb{F}_p$ is not $0$ or the inverse of $2 p_1$ mod $p$).

    By Lemma~\ref{prop:concentration}, taking $y$ and $Y$ sufficiently large, we have
    $$ \frac{1}{\log{Y}}\sum_{\ell\le Y}\frac{\left|f^{k_1}\left(2Wp_2^{m_2}\ell+r\right)-\chi_1(r)e^{iI_1(y,Y)+O\left(\varepsilon^{13}\right)}\right|}{\ell}\le\varepsilon^{12} $$
    and
    $$ \frac{1}{\log{Y}}\sum_{\ell\le Y}\frac{\left|g^{k_2}\left(4 p_1^{m_1}W\ell+\frac{2 p_1^{m_1}r+1}{p_2^{m_2}}\right)-e^{iI_2(y,Y)+O\left(\varepsilon^{13}\right)}\right|}{\ell}\le\varepsilon^{12}, $$
    where
    $$ I_1(y,Y) \vcentcolon= \sum_{y<p\le Y}{\frac{\mathfrak{I}{\left(f^{k_1}(p)\overline{\widetilde{\chi_1}(p)}\right)}}{p}} \quad ; \quad I_2(y,Y) \vcentcolon= \sum_{y<p\le Y}{\frac{\mathfrak{I}{\left(g^{k_2}(p)\overline{\widetilde{\chi_2}(p)}\right)}}{p}}. $$
    By a compactness argument, we can extract an increasing sequence $\left(Y_j\right)_{j \ge 1}$ of integers and $\left(\gamma_1,\gamma_2\right)\in\mathbb{T}^2$ depending on $y$ such that
    $$\left(e^{iI_1\left(y,Y_j\right)},e^{iI_2\left(y,Y_j\right)}\right)\xrightarrow[j\to+\infty]{}\left(\gamma_1,\gamma_2\right).$$
    Then, for $j$ sufficiently large in terms of $y$ and the underlying quantities, we have
    $$ \frac{1}{\log{Y_j}}\sum_{\ell\le Y_j}\frac{\left|f^{k_1}\left(2Wp_2^{m_2}\ell+r\right)-\chi_1(r)\gamma_1\right|}{\ell}\le2\varepsilon^{12} $$
    and
    $$ \frac{1}{\log{Y_j}}\sum_{\ell\le Y_j}\frac{\left|g^{k_2}\left(4 p_1^{m_1}W\ell+\frac{2 p_1^{m_1}r+1}{p_2^{m_2}}\right)-\gamma_2\right|}{\ell}\le2\varepsilon^{12}. $$

    Define
    $$
    \mathcal{I}_j \vcentcolon=
    \left\{
    \ell\le Y_j :
    \begin{cases}
        \left|f^{k_1}\left(2Wp_2^{m_2}\ell+r\right)-\chi_1(r)\gamma_1\right| \le 4\varepsilon^9 \\
        \left|g^{k_2}\left(4 p_1^{m_1}W\ell+\frac{2 p_1^{m_1}r+1}{p_2^{m_2}}\right)-\gamma_2\right| \le 4\varepsilon^9
    \end{cases}
    \right\}.
    $$
    By the union bound and by Rankin’s trick, we have
    \begin{equation*}
        \begin{aligned}
            \sum_{\substack{\ell \le Y_j \\ \ell\notin\mathcal{I}_j}} \frac{1}{\ell}
            &\le \frac{1}{4\varepsilon^9} \sum_{\ell\le Y_j}\frac{\left|f^{k_1}\left(2Wp_2^{m_2}\ell+r\right)-\chi_1(r)\gamma_1\right|}{\ell} \\
            &\quad + \frac{1}{4\varepsilon^9} \sum_{\ell\le Y_j}\frac{\left|g^{k_2}\left(4 p_1^{m_1}W\ell+\frac{2 p_1^{m_1}r+1}{p_2^{m_2}}\right)-\gamma_2\right|}{\ell} \\
            &\le \varepsilon^3\log{Y_j}.
        \end{aligned}
    \end{equation*}

    By our choice of $p_1$ and $p_2$, we can choose $m_1$ and $m_2$ such that
    $$ \left(f(p_1^{m_1}),g(p_2^{m_2})\right)=\left(z(\chi_1(r)\gamma_1)^{-1/k_1}\overline{f(2)},w\gamma_2^{-1/k_2}\right)+O\left(\varepsilon^9\right), $$
    where each root of unity can be chosen arbitrarily. Then for $\ell\in\mathcal{I}_j$ with $j$ sufficiently large, we have
    $$ f^{k_1}(n_\ell)=f^{k_1}(p_1^{m_1})f^{k_1}(2)f^{k_1}\left(2Wp_2^{m_2}\ell+r\right)=z^{k_1}+O\left(\varepsilon^9\right) $$
    and
    $$ g^{k_2}(n_\ell+1)=g^{k_2}\left(p_2^{m_2}\right)g^{k_2}\left(4 p_1^{m_1}W\ell+\frac{2 p_1^{m_1}r+1}{p_2^{m_2}}\right)=w^{k_2}+O\left(\varepsilon^9\right). $$
\end{proof}

We are now ready to prove our main theorem.

\begin{proof}[Proof of Theorem~\ref{thm:main}]
    Assume, for the sake of contradiction, that the conclusion of the theorem does not hold, and let us use the notation in the statements of Propositions \ref{prop:proofA} and \ref{prop:proofB}.
    
    Define
    $$ \mathcal{I}_j^\prime \vcentcolon= \left\{\ell\in\mathcal{I}_j : f(n_\ell)\notin B_{\varepsilon^3}\left(ze\left(\pm\frac{\varepsilon}{10}\right)\right), g(n_\ell+1)\notin B_{\varepsilon^3}\left(we\left(\pm\frac{\varepsilon}{10}\right)\right)\right\}. $$
    Taking $\varepsilon$ sufficiently small, we have
    \begin{equation*}
        \begin{multlined}
            \left\{\ell\in\mathcal{I}_j :
            \begin{cases}
                u^{k_1}\ell^{it_1k_1}\notin B_{k_1\varepsilon^3}\left(z^{k_1}e\left(\pm\frac{k_1\varepsilon}{10}\right)\right) \\
                v^{k_2}\ell^{it_2k_2}\notin B_{k_2\varepsilon^3}\left(w^{k_2}e\left(\pm\frac{k_2\varepsilon}{10}\right)\right)
            \end{cases}
            \right\}
            \subset\mathcal{I}_j^\prime.
        \end{multlined}
    \end{equation*}
    As a consequence, by the union bound and Lemma~\ref{lem:logsinglenit}, we obtain
    \begin{equation}\label{thm:main-1}
        \frac{1}{\log{Y_j}}\sum_{\ell\in\mathcal{I}_j^\prime}\frac{1}{\ell}\geq1-O\left(\varepsilon^3\right),
    \end{equation}
    where we used the fact that when $t_1=0$ (resp. $t_2=0$), then $u=z$ (resp. $v=w$).

    Now, define
    $$ S_j \vcentcolon= \frac{1}{\log{Y_j}}\sum_{\ell\in\mathcal{I}_j^\prime}\frac{\mathbb{1}_{B_{\frac{\varepsilon}{10}}(z)}\left(f(n_\ell)\right)\mathbb{1}_{B_\frac{\varepsilon}{10}(w)}\left(g(n_\ell+1)\right)}{\ell}. $$
    Let $X \ge 1$. By Lemma~\ref{lem:fourierapprox} and by definition of $\mathcal{I}_j^\prime$, we have
    $$ S_j=\sum_{\left|a\right|,\left|b\right|\le X}{\frac{\sin{\left(2\pi a\frac{\varepsilon}{10}\right)}}{\pi a z^a}\frac{\sin{\left(2\pi b\frac{\varepsilon}{10}\right)}}{\pi b w^b}\frac{1}{\log{Y_j}}\sum_{\ell\in\mathcal{I}_j^\prime}\frac{f^a(n_\ell)g^b(n_\ell+1)}{\ell}}+O\left(\frac{1}{\varepsilon^3X}\right). $$
    Let $\delta>0$. Take $j$ sufficiently large in terms of $X$ and $\delta$. Apply \eqref{thm:main-1}, then Lemmas~\ref{lem:nonpret} and \ref{lem:logtao} to obtain
    \begin{equation*}
        \begin{multlined}
            S_j
            = \frac{1}{k_1k_2} \sum_{\substack{\left|c\right|\le\frac{X}{k_1} \\ \left|d\right|\le\frac{X}{k_2}}} \frac{\sin{\left(2\pi c k_1\frac{\varepsilon}{10}\right)}}{\pi c z^{k_1c}}\frac{\sin{\left(2\pi d k_2\frac{\varepsilon}{10}\right)}}{\pi d w^{k_2d}}\frac{1}{\log{Y_j}}\sum_{\ell\le Y_j}\frac{f^{k_1c}(n_\ell)g^{k_2d}(n_\ell+1)}{\ell} \\
            + O\left(\frac{1}{\varepsilon^3X}+(\delta+\varepsilon^3) \left(\log{X}\right)^2\right).
        \end{multlined}
    \end{equation*}
    By Propositions \ref{prop:proofA} and \ref{prop:proofB}, we deduce that
    \begin{equation*}
        \begin{multlined}
            S_j
            = \frac{1}{k_1k_2} \sum_{\substack{\left|c\right|\le\frac{X}{k_1} \\ \left|d\right|\le\frac{X}{k_2}}} \frac{\sin{\left(2\pi c k_1\frac{\varepsilon}{10}\right)}}{\pi c z^{k_1c}}\frac{\sin{\left(2\pi d k_2\frac{\varepsilon}{10}\right)}}{\pi d w^{k_2d}} \\
            \qquad\qquad\qquad\quad \times \frac{1}{\log{Y_j}}\sum_{\ell\in\mathcal{I}_j^\prime}\frac{\left(u\ell^{it_1}\right)^{k_1c}\left(v\ell^{it_2}\right)^{k_2d}e^{O\left(X\varepsilon^9\right)}}{\ell} \\
            + O\left(\frac{1}{\varepsilon^3X}+(\delta+\varepsilon^3)\ \left(\log{X}\right)^2\right).
        \end{multlined}
    \end{equation*}
    Assume that $X \le \varepsilon^{-9}$. Then, it follows that
    \begin{equation}\label{thm:main-2}
        \begin{multlined}
            S_j
            = \frac{1}{k_1k_2\log{Y_j}} \sum_{\ell\in\mathcal{I}_j^\prime} \frac{1}{\ell} \sum_{\substack{\left|c\right|\le\frac{X}{k_1} \\ \left|d\right|\le\frac{X}{k_2}}} \frac{\sin{\left(2\pi c k_1\frac{\varepsilon}{10}\right)}}{\pi c}\frac{\sin{\left(2\pi d k_2\frac{\varepsilon}{10}\right)}}{\pi d}\left(\frac{u\ell^{it_1}}{z}\right)^{k_1c}\left(\frac{v\ell^{it_2}}{w}\right)^{k_2d} \\
            + O\left(\frac{1}{\varepsilon^3X}+\left(\delta+\varepsilon^3+X\varepsilon^9\right)\left(\log{X}\right)^2\right).
        \end{multlined}
    \end{equation}

    Now, define
    $$ \mathcal{I}_j^{\prime\prime} \vcentcolon= \left\{\ell\in\mathcal{I}_j^\prime : \left(\frac{u\ell^{it_1}}{z}\right)^{k_1}\notin B_{\varepsilon^3}\left(e\left(\pm\frac{k_1\varepsilon}{10}\right)\right), \left(\frac{v\ell^{it_2}}{w}\right)^{k_2}\notin B_{\varepsilon^3}\left(e\left(\pm\frac{k_2\varepsilon}{10}\right)\right)\right\}. $$
    By the union bound and Lemma~\ref{lem:logsinglenit}, we have
    \begin{equation}\label{thm:main-3}
        \frac{1}{\log{Y_j}}\sum_{\ell\in\mathcal{I}_j^{\prime\prime}}\frac{1}{\ell}\geq1-O\left(\varepsilon^3\right).
    \end{equation}
    Take $X=\varepsilon^{-6}$ and $\delta=\varepsilon^3$. When $j$ is sufficiently large in terms of $\varepsilon$, applying Lemma~\ref{lem:fourierapprox} again, we obtain by \eqref{thm:main-2} and \eqref{thm:main-3},
    \begin{align}\label{thm:main-4}
        S_j
        &= \frac{1}{k_1k_2\log{Y_j}}\sum_{\ell\in\mathcal{I}_j^{\prime\prime}}\frac{\mathbb{1}_{B_\frac{k_1\varepsilon}{10}\left(\left(\frac{z}{u}\right)^{k_1}\right)}\left(\ell^{it_1k_1}\right)\mathbb{1}_{B_\frac{k_2\varepsilon}{10}\left(\left(\frac{w}{v}\right)^{k_2}\right)}\left(\ell^{it_2k_2}\right)}{\ell}+O\left(\varepsilon^3\left(\log{\frac{1}{\varepsilon}}\right)\right)^2 \nonumber \\
        &\ge \frac{1}{k_1k_2\log{Y_j}}\sum_{\ell\le Y_j}\frac{\mathbb{1}_{B_\frac{\varepsilon}{10}\left(\frac{z}{u}\right)}\left(\ell^{it_1}\right)\mathbb{1}_{B_\frac{\varepsilon}{10}\left(\frac{w}{v}\right)}\left(\ell^{it_2}\right)}{\ell}+O\left(\varepsilon^3\left(\log{\frac{1}{\varepsilon}}\right)^2\right).
    \end{align}
    Let us deduce from \eqref{thm:main-4} that $S_j\gg\varepsilon^2$.
    
    If we have $s_1t_1=s_2t_2\neq0$ for some $s_1, s_2\in\mathbb{N}$ coprime, then we know that $\left(\frac{z}{u}\right)^{s_1}=\left(\frac{w}{v}\right)^{s_2}+O\left(\varepsilon^9\right)$. Hence, for some choices of $\mu,\nu\in\mathbb{T}$ such that $\mu^{s_2}=\frac{z}{u}$ and $\nu^{s_1}=\frac{w}{v}$, we have $\mu=\nu+O\left(\varepsilon^2\right)$. For $\varepsilon$ sufficiently small, it follows that
    \begin{equation*}
        \begin{aligned}
            S_j
            &\geq \frac{1}{k_1k_2\log{Y_j}}\sum_{\ell\le Y_j}\frac{\mathbb{1}_{B_\frac{\varepsilon}{10\left|s_2\right|}\left(\mu\right)}\left(\ell^{i\frac{t_1}{s_2}}\right)\mathbb{1}_{B_\frac{\varepsilon}{10\left|s_1\right|}\left(\nu\right)}\left(\ell^{i\frac{t_1}{s_2}}\right)}{\ell}+O\left(\varepsilon^2\log{\frac{1}{\varepsilon}}\right) \\
            &\ge \frac{1}{k_1k_2\log{Y_j}}\sum_{\ell\le Y_j}\frac{\mathbb{1}_{B_\frac{\varepsilon}{20\max{\left(\left|s_1\right|,\left|s_2\right|\right)}}\left(\mu\right)}\left(\ell^{i\frac{t_1}{s_2}}\right)}{\ell}+O\left(\varepsilon^2\log{\frac{1}{\varepsilon}}\right) \\
            &\gg \varepsilon
        \end{aligned}
    \end{equation*}
    by Lemma~\ref{lem:logsinglenit}.
    
    If $t_1=0$ and $t_2\neq0$ (the same argument holds if $t_2=0$ and $t_1\neq0$), then we have $u=z$, and we apply Lemma~\ref{lem:logsinglenit}.
    
    If $t_1, t_2\neq0$ and $\left(t_1,t_2\right)$ is linearly independent over $\mathbb{Q}$, then we conclude by Lemma~\ref{lem:logdoublenit}.
    
    If $t_1=t_2=0$, since $u=z$ and $v=w$, we immediately have $S_j\gg1$.

    Therefore, we have shown that in all cases,
    $$ \frac{1}{\log{Y_j}} \sum_{\substack{\ell \le Y_j \\ n_\ell\in\mathcal{A}\left(W^\prime Y_j+n_0\right)}} \frac{1}{\ell} \gg\varepsilon^2. $$
    This implies that, for $j$ sufficiently large,
    \begin{equation*}
        \begin{aligned}
            \frac{1}{\log{\left(W^\prime Y_j+n_0\right)}} \sum_{n\in\mathcal{A}\left(W^\prime Y_j+n_0\right)} \frac{1}{n}
            &\ge \frac{1}{\log{\left(2W^\prime Y_j\right)}} \sum_{\substack{\ell \le Y_j \\ n_\ell\in\mathcal{A}\left(W^\prime Y_j+n_0\right)}} \frac{1}{n_\ell} \\
            &\gg \frac{1}{W'\log{Y_j}} \sum_{\substack{\ell \le Y_j \\ n_\ell\in\mathcal{A}\left(W^\prime Y_j+n_0\right)}} \frac{1}{\ell} \\
            &\gg \frac{\varepsilon^2}{W^\prime},
        \end{aligned}
    \end{equation*}
    contradicting the fact that
    $$ \frac{1}{\log{\left(W^\prime Y_j+n_0\right)}} \sum_{n\in\mathcal{A}\left(W^\prime Y_j+n_0\right)} \frac{1}{n} \xrightarrow[j\to+\infty]{} 0. $$
\end{proof}

\section{Proof of Proposition~\ref{prop:counterex}}\label{sec:counterex}

First of all, note that if $f \in \overline{\mathcal{M}}$ satisfies \eqref{prop:counterex-0-1}, then it also satisfies \eqref{prop:counterex-0-2} by partial summation. Thus, it suffices to prove the first part of Proposition~\ref{prop:counterex}.

Next, we shall construct a multiplicative function $f:\mathbb{N}\to\mathbb{C}$ such that $f(n)$ is close to $f(n+1)$ for most integers $n$. To achieve this, we construct $f$ so that $f(n)$ can be approximated by some $n^{it}$ on long intervals of integers $n$. It is possible to construct such long enough intervals of $n$. However, it is worth noting that the construction given below does not extend to logarithmic averages, since these intervals are not sufficiently long on the logarithmic scale.

We use the function $f$ constructed by Klurman, Mangerel and Ter\"av\"ainen in \cite[Proposition 1.3]{klurman2023elliott}. For the sake of completeness, we recall their construction.

Set $t_1=100$ and $f(p)=1$ for $p\le t_1$. Recursively, assume that $f(p)$ is defined for all $p\le t_m$. By Kronecker’s theorem, the set $\left\{\left(p^{is}\right)_{p\le t_m} : s>e^{t_m}\right\}$ is dense in $\prod_{p\le t_m}\mathbb{T}$. It follows that there exists $s_{m+1}>e^{t_m}$ such that
$$ \forall p\le t_m,\ \left|f(p)-p^{is_{m+1}}\right|\le\frac{1}{t_m^2}. $$
Define $t_{m+1}=e^{s_{m+1}}$, and for $t_m<p\le t_{m+1}$, set $f(p)=p^{is_{m+1}}$.

This defines a completely multiplicative function $f:\mathbb{N}\rightarrow\mathbb{T}$, as well as two sequences $\left(s_m\right)$ and $\left(t_m\right)$ that tend to $+\infty$ as $m\rightarrow+\infty$. Moreover, it is straightforward to verify that $f$ has dense image, and that for all $K \in \mathbb{N}$, $(f^K,f^K) \notin \mathcal{F} \cup \mathcal{G}$.

Let $n\in\left[s_{m+1}^2, t_{m+1}\right]$ be such that, whenever $p^\ell|n$ for some $p\le t_m$, we have $\ell\le10\log{t_m}$. Then,
\begin{equation*}
    \begin{aligned}
        f(n)
        &= n^{is_{m+1}} \left( 1+O\left(\frac{1}{t_m^2}\right) \right)^{\sum_{\substack{p \le t_m \\ p^\ell||n}} \ell}
        = n^{is_{m+1}}\left(1+O\left(\frac{1}{t_m^2}\right)\right)^{O\left(t_m\right)} \\
        &= n^{is_{m+1}}\left(1+O\left(\frac{1}{t_m}\right)\right).
    \end{aligned}
\end{equation*}
Hence, for all $n\in\left[s_{m+1}^2, t_{m+1}-1\right]$ such that, whenever $p^\ell|n(n+1)$ for some $p\le t_m$, we have $\ell\le10\log{t_m}$, we obtain
\begin{equation*}
    \begin{aligned}
        f(n)\overline{f(n+1)}
        &= \left(1+\frac{1}{n+1}\right)^{is_{m+1}}\left(1+O\left(\frac{1}{t_m}\right)\right) \\
        &= \left(1+O\left(\frac{s_{m+1}}{s_{m+1}^2}\right)\right)\left(1+O\left(\frac{1}{t_m}\right)\right)
        = 1+o(1).
    \end{aligned}
\end{equation*}
Further, the sum $\sum_n \frac{1}{n}$ over integers $n\le t_{m+1}-1$ such that there exists $p^\ell|n(n+1)$ with $p\le t_m$ and $\ell>10\log{t_m}$ is
$$
\le \sum_{\substack{p \le t_m \\ \ell>10\log{t_m}}} \sum_{\substack{n \le t_{m+1}-1 \\p^\ell|n(n+1)}} \frac{1}{n}
\le 2\sum_{p\le t_m}\sum_{\ell>10\log{t_m}}\frac{\log \left(\frac{t_{m+1}}{p^\ell}+1\right)}{p^\ell}
= o\left(\log t_{m+1}\right).
$$
It follows
$$
\frac{1}{\log t_{m+1}} \sum_{\substack{n \le t_{m+1} \\ \left|f(n)\overline{f(n+1)}+1\right|\le\frac{1}{3}}} \frac{1}{n}
\le \frac{\log s_{m+1}^2}{\log t_{m+1}}+o(1)
\rightarrow 0.
$$

\section{Some consequences of Theorem \ref{thm:main}}\label{sec:corollaries}

The aim of this section is to prove two corollaries of Theorem~\ref{thm:main}. First, Corollary~\ref{cor:DKP2} generalizes the following conjecture of De Koninck, K\'atai and Phong \cite[Conjecture 4]{de2019some} in the case of completely multiplicative functions.

\begin{conjecture}
    Let $f:\mathbb{N}\to\mathbb{T}$ be a multiplicative function such that there exist some $w \in \mathbb{T}$ and some $\varepsilon>0$ such that
    $$ \lim\limits_{x\to+\infty} \frac{1}{x} \sum_{\substack{n \le x \\ |f(n)\overline{f(n+1)}-w| \le \varepsilon}} 1 = 0. $$
    Then, there exist $K \ge 1$ and $t \in \mathbb{R}$ such that $f^K(n)=n^{it}$ for all $n \in \mathbb{N}$.
\end{conjecture}

\begin{proof}[Proof of Corollary~\ref{cor:DKP2}]
    If the image of $f$ is not dense, then it is finite. Indeed, in this case, if we write $f(p)=e(\alpha_p)$ for all primes $p$, then $(\alpha_p)_p$ has to be a sequence of rational numbers with bounded denominators. Then, there exists $K \in \mathbb{N}$ such that $f^K=1$.
    
    Assume from now on that the image of $f$ is dense in $\mathbb{T}$. The hypothesis implies
    $$ \lim\limits_{x\to+\infty} \frac{1}{\log x} \sum_{\substack{n \le x \\ |f(n)-w|, |f(n+1)-1| \le \frac{\varepsilon}{2}}} \frac{1}{n} = 0. $$
    By Theorem~\ref{thm:main}, it suffices to show that $(f^K,f^K) \notin \mathcal{F}$ for all $K \in \mathbb{N}$. Let $K \in \mathbb{N}$, assume by contradiction that there exist a prime $p$ and an irrational number $\alpha$ such that $f(p)=e(\alpha)$ and $f^K(q)=1$ for all primes $q \neq p$. Assume further that $K$ is minimal with this property. Then, by Lemma~\ref{lem:nonpret}, for all $A$ sufficiently large and for all $x\geq x_{0}=x_{0}\left(A\right)$, $f$ is $(A,x)$-non-pretentious. Thus, by Lemma~\ref{lem:logtao}, for all integers $j \in [1,K-1]$, we have
    \begin{equation*}
        \begin{multlined}
            \frac{1}{\log{x}}\sum_{\substack{n\leq x \\ p^{k}||n}} \frac{{\left(f\left(\frac{n}{p^{k}}\right)\overline{f(n+1)}\right)}^{j}}{n} \\
            = \sum_{r=1}^{p-1} \frac{1}{\log{x}}\sum_{0\leq m\leq \frac{x-p^{k}r}{p^{k+1}}} \frac{{\left(f(pm+r)\overline{f(p^{k+1}m+p^{k}r+1)}\right)}^{j}}{p^{k+1}m+p^{k}r}
            \to 0.
        \end{multlined}
    \end{equation*}
    Fix $k \in \mathbb{N}$ such that $|e(k\alpha)-w|\le\varepsilon$. It follows by orthogonality in ${\mu }_{K}$ that
    \begin{equation*}
        \begin{aligned}
            \frac{1}{\log x} \sum_{\substack{n \le x \\ |f(n)\overline{f(n+1)}-w| \le \varepsilon}} \frac{1}{n}
            &\ge \frac{1}{\log{x}}\sum_{\substack{n\leq x \\ p^{k}||n \\ f\left(\frac{n}{p^{k}}\right)\overline{f(n+1)}=1}} \frac{1}{n} \\
            &= \frac{1}{K}\sum_{j=0}^{K-1} \frac{1}{\log x} \sum_{\substack{n\leq x \\ p^{k}||n}} \frac{{\left(f\left(\frac{n}{p^{k}}\right)\overline{f(n+1)}\right)}^{j}}{n} \\
            &\xrightarrow[x\to+\infty]{} \frac{1}{p^{k}}\left(1-\frac{1}{p}\right) > 0,
        \end{aligned}
    \end{equation*}
    which contradicts the hypothesis.
\end{proof}

\begin{remark}
    The statement of Corollary~\ref{cor:DKP2} is optimal. Indeed, let $f$ be a completely multiplicative function such that $f^K(n)=n^{it}$ for some minimal $K \ge 1$ and some $t \in \mathbb{R}$. Then, write $f(n)=h(n)n^{it/K}$. Note that $n^{it/K}(n+1)^{-it/K}\xrightarrow[n\to+\infty]{}1$, and that $h$ is a function taking values in the set $\mu_K$ of $K$-th roots of unity. Thus, taking $w\in\mathbb{T}\backslash\mu_K$ and $\varepsilon>0$ sufficiently small, we have
    $$
    \frac{1}{x} \sum_{\substack{n \le x \\ |f(n)\overline{f(n+1)}-w| \le \varepsilon}} 1
    \le \frac{1}{x} \left( \sum_{\substack{n \le x \\ |h(n)\overline{h(n+1)}-w| \le \varepsilon}} 1 + O(1) \right)
    = O\left(\frac{1}{x}\right).
    $$
\end{remark}

Next, Corollary~\ref{cor:katai2} is motivated by the following problem. Let $f:\mathbb{N}\to\mathbb{T}$ be a multiplicative function such that $f(n)-f(n+1) \xrightarrow[n\to+\infty]{}0$. Then, do we have $f(n)=n^{it}$ for some $t\in\mathbb{R}$? In 1996, Wirsing \cite{wirsing1996conjecture} proved that this is true. In addition, it is an exercise to deduce from this that, if $f,g:\mathbb{N}\to\mathbb{T}$ are two completely multiplicative functions such that $f(n)-g(n+1) \xrightarrow[n\to+\infty]{}0$, then $f(n)=g(n)=n^{it}$ for some $t\in\mathbb{R}$. In 1983, K\'atai \cite{katai1983some} conjectured a strengthened version of Wirsing's result, assuming that $\lim\limits_{x\to+\infty} \frac{1}{\log x} \sum_{n \le x} \frac{|f(n)-f(n+1)|}{n} = 0$. It was proved by Klurman \cite{klurman2017correlations} that Wirsing's conclusion holds with this weaker hypothesis. Then, it is natural to ask whether $\lim\limits_{x\to+\infty} \frac{1}{\log x} \sum_{n \le x} \frac{|f(n)-g(n+1)|}{n} = 0$ implies $f(n)=g(n)=n^{it}$ for some $t\in\mathbb{R}$. It turns out that this is not a direct corollary of Klurman's result. We now show that the statement nevertheless holds.

\begin{proof}[Proof of Corollary~\ref{cor:katai2}]
    First, assume that $f$ and $g$ do not have dense images. Then, there exists an integer $K\ge 1$ such that $f^{K}=g^{K}=1$. If $f=g=1$, then there is nothing to do. Thus, assume that $f\neq 1$ and consider $z\in {\mu }_{K}\backslash \{1\}$. Then, the hypothesis implies that
    \begin{equation}\label{cor:katai2-01}
        \lim\limits_{x\to+\infty} \frac{1}{\log{x}}\sum_{\substack{n\leq x \\ f(n)=z \\ g(n+1)=1}} \frac{1}{n} = 0.
    \end{equation}
    On the other hand, by orthogonality in ${\mu }_{K}$, we have
    \begin{equation}\label{cor:katai2-02}
        \frac{1}{\log{x}}\sum_{\substack{n\leq x \\ f(n)=z \\ g(n+1)=1}} \frac{1}{n}=\frac{1}{K^{2}}\sum_{a=0}^{K-1} \frac{1}{z^{a}}\sum_{b=0}^{K-1} \frac{1}{\log{x}}\sum_{n\leq x} \frac{f^{a}\left(n\right)g^{b}(n+1)}{n}.
    \end{equation}
    By Lemmas \ref{lem:nonpret} and \ref{lem:logtao}, the only term that does not vanish in \eqref{cor:katai2-02} as $x\to+\infty$ is the one that corresponds to $(a,b)=(0,0)$. Thus,
    $$ \frac{1}{\log{x}}\sum_{n\leq x} \frac{\left|f(n)-g(n+1)\right|}{n}\geq \frac{|z-1|}{\log{x}}\sum_{\substack{n\leq x \\ f(n)=z \\ g(n+1)=1}} \frac{1}{n}\rightarrow \frac{|z-1|}{K^{2}} > 0, $$
    which contradicts \eqref{cor:katai2-01}.

    Now, assume that $f$ has a dense image, and that there exists an integer $K \ge 1$ such that $g^{K}=1$. Then, by the elementary inequality
    $$ \left|f^{K}(n)-g^{K}(n+1)\right|\leq K\left|f(n)-g(n+1)\right|, $$
    the hypothesis implies
    $$ \lim\limits_{x\to+\infty} \frac{1}{\log{x}}\sum_{n\leq x} \frac{\left|f^{K}(n)-1\right|}{n} = 0. $$
    Since
    $$ \sum_{n\leq x} \frac{\left|f^{K}(n)-1\right|}{n}\geq \frac{1}{3}\sum_{\substack{n\leq x \\ \left|f^{K}(n)-1\right|> \frac{1}{3}}} \frac{1}{n}\geq \frac{1}{3}\sum_{\substack{n\leq x \\ \left|f^{K}(n)+1\right|\leq \frac{1}{3}}} \frac{1}{n}, $$
    we deduce that
    \begin{equation}\label{cor:katai2-03}
        \lim\limits_{x\to+\infty} \frac{1}{\log{x}}\sum_{\substack{n\leq x \\ \left|f^{K}(n)+1\right|\leq \frac{1}{3}}} \frac{1}{n} = 0.
    \end{equation}
    Now, note that one can always choose a function $g\in \overline{\mathcal{M}}$ such that for all $K_{1},K_{2}\geq 1$, we have $\left(f^{KK_{1}},g^{K_{2}}\right)\notin \mathcal{F}\cup \mathcal{G}$. Thus, \eqref{cor:katai2-03} contradicts Theorem~\ref{thm:main}.

    From now on, we assume that $f$ and $g$ both have dense images. Let $\varepsilon>0$. Then, by Rankin’s trick, we have
    $$
    \frac{1}{\log{x}} \sum_{\substack{n \le x \\ \left|f(n)-g(n+1)\right|>\varepsilon}} \frac{1}{n}
    \le \frac{1}{\varepsilon\log{x}}\sum_{n\le x}\frac{\left|f(n)-g(n+1)\right|}{n}.
    $$
    In particular, taking $\varepsilon$ sufficiently small, we have
    $$ \lim\limits_{x\to+\infty}  \frac{1}{\log{x}} \sum_{\substack{n \le x \\ \left|f(n)-1\right|,\left|g(n+1)+1\right|\le\varepsilon}} \frac{1}{n} = 0. $$
    By Theorem~\ref{thm:main}, there exists an integer $K\ge1$ such that $\left(f^K,g^K\right)\in\mathcal{F}\cup\mathcal{G}$.

    Let us first assume that $\left(f^K,g^K\right)\in\mathcal{F}$. As before, the hypothesis implies
    \begin{equation}\label{cor:katai2-1}
        \lim\limits_{x\to+\infty} \frac{1}{\log{x}}\sum_{n\le x}\frac{\left|f^K(n)-g^K(n+1)\right|}{n} = 0.
    \end{equation}
    By definition of $\mathcal{F}$, $f^K(p)$ and $g^K(p)$ are equal to $1$ at all primes $p$ except for a single prime $p=p_0$. Let $\alpha\in\mathbb{R}\backslash\mathbb{Q}$ be such that $f^K\left(p_0\right)=e(\alpha)$. Restricting the summation in \eqref{cor:katai2-1} to integers $n$ such that $p_0||n$, in which case $\left|f^K(n)-g^K(n+1)\right|=\left|e(\alpha)-1\right|>0$, we have that the left-hand side of \eqref{cor:katai2-1} is $\geq\left|e(\alpha)-1\right|\left(\frac{1}{p_0}-\frac{1}{p_0^2}\right)>0$, which yields a contradiction. As a consequence, we have $\left(f^K,g^K\right)\in\mathcal{G}$.

    It remains to prove that $K$ can be taken to be equal to $1$. Take $K\ge1$ minimal such that $\left(f^K,g^K\right)\in\mathcal{G}$, and let us prove that $K=1$. Let $t\in\mathbb{R}$ be such that $f^K(n)=g^K(n)=n^{iKt}$. Write $f(n)=f_0(n)n^{it}$ and $g(n)=g_0(n)n^{it}$, where $f_0$ and $g_0$ are completely multiplicative functions taking values in $\mu_K$. From the hypothesis, and using the observation that $n^{it}(n+1)^{-it}=1+O_t\left(\frac{1}{n}\right)$, we have
    $$ \lim\limits_{x\to+\infty} \frac{1}{\log x} \sum_{n \le x} \frac{|f_0(n)-g_0(n+1)|}{n} = 0. $$
    This implies that, for all $z,w\in\mu_K$ distinct, we have
    \begin{equation}\label{cor:katai2-2}
        \lim\limits_{x\to+\infty} S(x;z,w) = 0,
    \end{equation}
    where
    $$ S(x;z,w) := \frac{1}{\log{x}} \sum_{\substack{n \le x \\ f_0(n)=z \\ g_0(n+1)=w}} \frac{1}{n}. $$
    Next, let $K_1$ and $K_2$ be minimal positive integers such that $f_0^{K_1}=\widetilde{\chi_1}$ and $g_0^{K_2}=\widetilde{\chi_2}$, for some Dirichlet characters $\chi_1$ and $\chi_2$. Let $q$ be the period of $\chi_2$. By orthogonality in $\mu_K$, we have
    \begin{equation*}
        \begin{aligned}
            S(x;z,w)
            &\ge \frac{1}{q\log x} \sum_{\substack{m \le \frac{x}{q} \\ f_0(mq)=z \\ g_0(mq+1)=w}} \frac{1}{m} \\
            &=\frac{1}{K^2q}\sum_{a\le K}\sum_{b\le K}{f_0^a(q)z^{-a}w^{-b}\frac{1}{\log{x}}\sum_{m\le\frac{x}{q}}\frac{f_0^a(m)g_0^b(mq+1)}{m}}.
        \end{aligned}
    \end{equation*}
    Let $\varepsilon>0$. By Lemmas \ref{lem:nonpret} and \ref{lem:logtao}, for $x$ sufficiently large in terms of $\varepsilon$, we have
    $$ S(x;z,w)\geq\frac{1}{K^2q}\sum_{c\le\frac{K}{K_1}}\sum_{d\le\frac{K}{K_2}}{{\widetilde{\chi_1}}^c(q)z^{-K_1c}w^{-K_2d}\frac{1}{\log{x}}\sum_{m\le\frac{x}{q}}\frac{{\widetilde{\chi_1}}^c(m)}{m}}+O\left(\varepsilon\right), $$
    where we used the fact that $\widetilde{\chi_2}(mq+1)=1$. Denote by $r$ the order of $\chi_1$. By \cite[Lemma 2.3.1]{granvillemultiplicative}, ${\widetilde{\chi_1}}^c$ does not pretend to any function of the form $n^{it}$ ($t\in\mathbb{R}$) when $r \nmid c$. Thus, by Hal\'asz's theorem (see e.g. \cite[Lemma 2.1.11]{granvillemultiplicative}), we have
    \begin{equation*}
        \begin{aligned}
            S(x;z,w)
            &\ge \frac{1}{K^2q}\sum_{e\le\frac{K}{K_1r}}\sum_{d\le\frac{K}{K_2}}{z^{-K_1re}w^{-K_2d}\frac{1}{\log{x}}\sum_{m\le\frac{x}{q}}\frac{1}{m}}+O\left(\varepsilon\right) \\
            &= \frac{\mathbb{1}_{z^{K_1r}=1}\mathbb{1}_{w^{K_2}=1}}{KK_2}+O\left(\varepsilon\right),
        \end{aligned}
    \end{equation*}
    since $K_1r,K_2|K$. Then, taking $\varepsilon$ sufficiently small, \eqref{cor:katai2-2} implies that for all $z\neq w$ in $\mu_K$, at least one of $z\notin\mu_{K_1r}$ or $w\notin\mu_{K_2}$ holds. This means that $\mu_{K_1r}\cup\mu_{K_2}=\left\{1\right\}$. Thus, $K_1=K_2=1$ and $r=1$. By symmetry, if $s$ denotes the order of $\chi_2$, then $s=1$. Moreover, it is straightforward to see that $K_1r$ is the minimal integer such that $f_0^{K_1r}=1$, and that $K_2s$ is the minimal integer such that $g_0^{K_2s}=1$. By minimality of $K$, we have $K=\operatorname{lcm}{(K_1r,K_2s)}=1$.
\end{proof}

\bibliographystyle{plain}
\bibliography{bibliography}

@article{tao2016logarithmically,
  title={The logarithmically averaged {C}howla and {E}lliott conjectures for two-point correlations},
  journal={Forum of Mathematics, Pi},
  author={Tao, Terence},
  volume={4},
  year={2016},
  publisher={Cambridge University Press}
}

@article{klurman2018orbits,
  title={On the orbits of multiplicative pairs},
  author={Klurman, Oleksiy and Mangerel, Alexander P},
  journal={Algebra {\&} Number Theory},
  volume={14},
  number={1},
  year={2020}
}

@article{khale2024explicit,
  title={An explicit {V}inogradov--{K}orobov zero-free region for {D}irichlet {$L$}-functions},
  author={Khale, Tanmay},
  journal={The Quarterly Journal of Mathematics},
  volume={75},
  number={1},
  pages={299--332},
  year={2024},
  publisher={Oxford University Press UK}
}

@misc{granvillemultiplicative,
  title={Multiplicative number theory: The pretentious approach. 2014. Opublicerat manuskript},
  author={Granville, Andrew and Soundarajan, Kannan}
}

@article{daroczy1989characterization,
  title={Characterization of additive functions with values in the circle group},
  author={Dar\'oczy, Zolt\'an and K\'atai, Imre},
  journal={Publicationes Mathematicae Debrecen},
  volume={36},
  number={1--4},
  year={1989},
  pages={1--7}
}

@article{de2021variations,
  title={On the variations of completely multiplicative functions at consecutive arguments},
  author={De Koninck, Jean-Marie and K{\'a}tai, Imre and Phong, Bui Minh},
  journal={Publicationes Mathematicae Debrecen},
  volume={98},
  number={1-2},
  pages={219--230},
  year={2021}
}

@inproceedings{de2019some,
  title={On some consequences of recently proved conjectures},
  author={De Koninck, Jean-Marie and K{\'a}tai, Imre and Phong, Bui Minh},
  booktitle={Annales Universitatis Scientiarum Budapestinensis de Rolando E{\"o}tv{\"o}s Nominatae. Sectio Computatorica},
  volume={49},
  year={2019}
}

@article{klurman2021multiplicative,
  title={Multiplicative functions that are close to their mean},
  author={Klurman, Oleksiy and Mangerel, Alexander and Pohoata, Cosmin and Ter{\"a}v{\"a}inen, Joni},
  journal={Transactions of the American Mathematical Society},
  volume={374},
  number={11},
  pages={7967--7990},
  year={2021}
}

@article{klurman2023elliott,
  title={On {E}lliott's conjecture and applications},
  author={Klurman, Oleksiy and Mangerel, Alexander P and Ter{\"a}v{\"a}inen, Joni},
  journal={arXiv:2304.05344},
  year={2023}
}

@article{wirsing1996conjecture,
  title={On a conjecture of {K}{\'a}tai for additive functions},
  author={Wirsing, Eduard and Yuan-Sheng, Tang and Pin-Tsung, Shao},
  journal={Journal of Number Theory},
  volume={56},
  number={2},
  pages={391--395},
  year={1996},
  publisher={Elsevier}
}

@article{klurman2017correlations,
  title={Correlations of multiplicative functions and applications},
  author={Klurman, Oleksiy},
  journal={Compositio Mathematica},
  volume={153},
  number={8},
  pages={1622--1657},
  year={2017},
  publisher={London Mathematical Society}
}

@article{katai1983some,
  title={Some results and problems on arithmetical functions},
  author={K{\'a}tai, Imre},
  journal={Studia Sci. Math. Hungar.},
  volume={16},
  pages={289--295},
  year={1983}
}

@article{charamaras2025multiplicative,
  title={On multiplicative recurrence along linear patterns},
  author={Charamaras, Dimitrios and Mountakis, Andreas and Tsinas, Konstantinos},
  journal={Journal of the London Mathematical Society},
  volume={112},
  number={3},
  pages={e70292},
  year={2025},
  publisher={Wiley Online Library}
}

@article{klurman2018rigidity,
  title={Rigidity theorems for multiplicative functions},
  author={Klurman, Oleksiy and Mangerel, Alexander P},
  journal={Mathematische Annalen},
  volume={372},
  number={1},
  pages={651--697},
  year={2018},
  publisher={Springer}
}

@article{leung2024multiplicative,
  title={Multiplicative recurrence of {M}\"{o}bius transformations},
  author={Leung, Sun-Kai and T{\'a}fula, Christian},
  journal={arXiv:2409.12936},
  year={2024}
}

@article{donoso2023additive,
  title={Additive averages of multiplicative correlation sequences and applications},
  author={Donoso, Sebasti{\'a}n and Le, Anh N and Moreira, Joel and Sun, Wenbo},
  journal={Journal d'Analyse Math{\'e}matique},
  volume={149},
  number={2},
  pages={719--761},
  year={2023},
  publisher={Springer}
}

\end{document}